\documentclass[journal]{IEEEtran}

\usepackage{amsmath,amssymb,amsfonts}
\usepackage{bm}
\usepackage{cite}
\usepackage{algorithm}
\usepackage{algorithmic}
\usepackage{graphicx}
\ifCLASSOPTIONcompsoc
\usepackage[caption=false,font=normalsize,labelfont=sf,textfont=sf]{subfig}
\else
\usepackage[caption=false,font=footnotesize]{subfig}
\fi
\usepackage{caption}
\usepackage{lipsum}
\usepackage{enumerate}
\usepackage{amsthm}
\usepackage{amsmath}
\usepackage{amssymb}
\usepackage{color}
\usepackage{url}
\usepackage{booktabs}

\usepackage{etoolbox}

\BeforeBeginEnvironment{figure}{\vskip-2ex}
\AfterEndEnvironment{figure}{\vskip-1ex}

\newtheorem{definition}{Definition}
\newtheorem{lemma}{Lemma}
\newtheorem{theorem}{Theorem}
\newtheorem{problem}{Problem}
\newtheorem{proposition}{Proposition}

\newtheorem{assumption}{Assumption}
\newtheorem{remark}{Remark}

\usepackage{setspace}
\usepackage{bm}
\usepackage{threeparttable}
 
\makeatletter
\newcommand*{\rom}[1]{\expandafter\@slowromancap\romannumeral #1@}
\makeatother

\ifCLASSOPTIONcompsoc
  \usepackage[caption=false,font=normalsize,labelfont=sf,textfont=sf]{subfig}
\else
  \usepackage[caption=false,font=footnotesize]{subfig}
\fi

\usepackage{tikz}
\newcommand\copyrighttext{%
  \footnotesize \textcopyright 2023 IEEE. Personal use of this material is permitted.
  Permission from IEEE must be obtained for all other uses, in any current or future
  media, including reprinting/republishing this material for advertising or promotional
  purposes, creating new collective works, for resale or redistribution to servers or
  lists, or reuse of any copyrighted component of this work in other works.}
\newcommand\copyrightnotice{%
\begin{tikzpicture}[remember picture,overlay]
\node[anchor=south,yshift=10pt] at (current page.south) {\fbox{\parbox{\dimexpr\textwidth-\fboxsep-\fboxrule\relax}{\copyrighttext}}};
\end{tikzpicture}%
}

\begin{document}

\title{Optimization Landscape of Policy Gradient Methods for Discrete-time Static Output Feedback}

\author{Jingliang Duan, Jie Li, Xuyang Chen, Kai Zhao, Shengbo Eben Li, Lin Zhao
\thanks{The work of J. Duan was supported in part by the NSF China under Grant 52202487 and in part by the State Key Laboratory of Automotive Safety and Energy, China under Project KFY2212. The work of L. Zhao was supported in part by the Singapore Ministry of Education Tier 1 Academic Research Fund (A0009030-00-00, 22-5460-A0001). J. Duan and J. Li contributed equally to this study. All correspondence should be sent to L. Zhao. 
}
\thanks{J. Duan is with the School of Mechanical Engineering, University of Science and Technology Beijing, Beijing, 100083, China, and also with the Department of Electrical and Computer Engineering, National University of Singapore, Singapore 119077. {\tt\small Email:duanjl@ustb.edu.cn}.}
\thanks{X. Chen, K. Zhao and L. Zhao are with the Department of Electrical and Computer Engineering, National University of Singapore, Singapore 119077. {\tt\small Email: chenxuyang@u.nus.edu, zhaokai@cqu.edu.cn, elezhli@nus.edu.sg}.}
\thanks{J. Li and S. E. Li are with the School of Vehicle and Mobility, Tsinghua University, Beijing, 100084, China. {\tt\small Email: jie-li18@mails.tsinghua.edu.cn, lishbo@mail.tsinghua.edu.cn}.
}
}

 \maketitle
  \copyrightnotice

\begin{abstract}
In recent times, significant advancements have been made in delving into the optimization landscape of policy gradient methods for achieving optimal control in linear time-invariant (LTI) systems. Compared with state-feedback control, output-feedback control is more prevalent since the underlying state of the system may not be fully observed in many practical settings. This paper analyzes the optimization landscape inherent to policy gradient methods when applied to static output feedback (SOF) control in discrete-time LTI systems subject to quadratic cost. We begin by establishing crucial properties of the SOF cost, encompassing coercivity, $L$-smoothness, and $M$-Lipschitz continuous Hessian. Despite the absence of convexity, we leverage these properties to derive novel findings regarding convergence (and nearly dimension-free rate) to stationary points for three policy gradient methods, including the vanilla policy gradient method, the natural policy gradient method, and the Gauss-Newton method. Moreover, we provide proof that the vanilla policy gradient method exhibits linear convergence towards local minima when initialized near such minima. The paper concludes by presenting numerical examples that validate our theoretical findings. These results not only characterize the performance of gradient descent for optimizing the SOF problem but also provide insights into the effectiveness of general policy gradient methods within the realm of reinforcement learning.
\end{abstract}

\begin{IEEEkeywords}
static output feedback, policy gradient, reinforcement learning
\end{IEEEkeywords}

\IEEEpeerreviewmaketitle

\section{Introduction}
\label{sec.introduction}

Reinforcement learning (RL) has showcased remarkable proficiency comparable to human capabilities in a variety of challenging tasks, spanning from games to robot control \cite{silver2016mastering,li2022reinforcement,duan2023relaxed,guan2022integrated}. RL methods relying on policy gradient, including DDPG \cite{lillicrap2015DDPG}, SAC\cite{Haarnoja2018SAC}, and DSAC\cite{duan2021distributional}, are commonly employed to identify parameterized optimal control policies for tasks with continuous action space. However, despite these achievements, a complete theoretical grasp of the complexity and performance of such algorithms remains lacking, even in fundamental scenarios like linear time-invariant (LTI) systems.

Optimal control problems have served as a potent tool for exploring various characteristics of RL, including aspects like stability \cite{wang2022intelligent,wang2022system}. Within the framework of policy gradient methods, prior investigations have delved into the optimization landscape and the attributes of convergence, particularly within the context of linear quadratic regulator (LQR) problems \cite{zhang2020global}. It is widely acknowledged that the optimal solution of LQR problems can be derived through the solution of the algebraic Riccati equation (ARE). However, in the pursuit of unveiling the characteristics of policy gradient during the training process, the focus shifts towards the direct optimization of the linear policy using the LQR cost, rather than solving the corresponding ARE.  In this context, it is noteworthy that the related optimization problem generally assumes a non-convex nature since the set of stabilizing state-feedback gains may lack convexity \cite{fazel2018global}. An influential work by Fazel \emph{et al.} (2018) discovered that the discrete-time LQR objective function exhibits properties of gradient dominance and almost smoothness, enabling policy gradient methods to achieve linear global convergence, despite the non-convexity of the LQR \cite{fazel2018global}. Subsequent studies have explored akin attributes, with specific attention to both discrete-time and continuous-time LQR \cite{bu2019lqr,bhandari2019global,mohammadi2019CT-LQR,malik2019derivative}, as well as various LQR variations \cite{hambly2020noisyLQR,zheng2021sampleLQG,ren2021lqrtracking,jansch2020Mjump,furieri2020distributedLQR,zhang2020robust}.

Compared with state-feedback control, output-feedback control is more common since the underlying state of the system may not be fully observed in practical settings \cite{wang2017observer, yu2021command,zhao2022composite,zong2022output}. Most of the existing convergence results of gradient descent are built on full state feedback, whereas the convergence for static output feedback (SOF) LQR has received little attention. As one of the most crucial open topics in LTI systems, SOF can only acquire some linear combinations of states, rather than entire states \cite{syrmos1997SOF}. Unlike the state feedback LQR, the output feedback gain of SOF may have a disconnected domain, with local minima, saddle points, or even local maxima in each component \cite{fatkhullin2020CTSOF,feng2020connectivity,bu2019topological}. Therefore, finding the globally optimal SOF controller using gradient descent is generally intractable. However, it is still of great significance to investigate the optimization landscape of SOF, particularly concerning the convergence towards stationary points, which will bring new insights into the performance of policy gradient methods for partially observed control problems. 

Recent efforts have elucidated the optimization landscape pertaining to continuous-time LTI systems in the context of SOF, delineating convergence rates to stationary points \cite{fatkhullin2020CTSOF}. However, these findings are limited to the vanilla policy gradient method, with the convergence behaviors of popular alternatives like the natural policy gradient \cite{kakade2001natural} and the Gauss-Newton method \cite{foresee1997gauss} yet to be fully clarified. Given the inherent distinctions between difference equations and differential equations, the analysis of optimization landscapes in discrete-time systems assumes a distinct character. Notably, discrete-time SOF holds practical significance due to its alignment with control frequency limitations, and the utilization of discrete-time data from real-world systems holds the promise of extending convergence insights to model-free contexts.

Despite these considerations, the theoretical properties of policy gradient methods applied to discrete-time SOF scenarios have been overlooked in existing studies. This study takes the initial step towards bridging this gap and offers the following principal contributions.
\begin{enumerate}
    \item We unveil several crucial properties of the SOF cost function, in spite of its non-convex nature. Notable among these properties are the compact sublevel set, $L$-smoothness, and $M$-Lipschitz continuous Hessian, which are instrumental in the subsequent theoretical analyses. A standout feature of our work is the establishment of Hessian Lipschitz continuity, a property that provides critical insights into the path of convergence towards local minima within SOF problems. This property is typically overlooked in the extant literature on both SOF and state-feedback LQR.  Diverging from approaches that establish $L$-smoothness  \cite{jansch2020Mjump}, we prove Hessian Lipschitz continuity through a direct application of its definition, thereby avoiding complex tensor operations.
    \item Unlike state-feedback LQR, where theories of convergence often hinge on the concept of gradient dominance \cite{fazel2018global,bu2019lqr,bhandari2019global,mohammadi2019CT-LQR,hambly2020noisyLQR,ren2021lqrtracking}, the landscape of SOF problems presents greater complexities. This complexity arises from the disconnected nature of the stabilizing SOF domain and the potential multiplicity of stationary points. Leveraging the compactness and $L$-smoothness of the SOF cost function, we show that three different policy gradient methods (the vanilla policy gradient, the natural policy gradient, and the Gauss-Newton method) can converge to stationary points at a (nearly) dimension-free rate, given an initial stabilizing policy.
    \item Furthermore, when the initial point is proximate to a local minimum, we demonstrate that the vanilla policy gradient method converges linearly towards it, predicated on the Lipschitz continuity of the Hessian.
\end{enumerate}

It is worth noting that the primary goal of this study is not the introduction of a new control algorithm for specific control problems. Rather, we focus on the SOF problem as a fertile ground for investigating the convergence, complexity, and optimality of policy gradient-based RL algorithms. Our findings offer new perspectives on the effectiveness of policy gradient methods in SOF problems and illuminate the efficacy of employing general policy gradient methods when learning SOF policies with unknown systems.

\textbf{Notation}
$\|M\|$, $\|M\|_F$, and $\rho(M)$ denote the induced 2-norm, Frobenius norm, and spectral radius of a matrix $M$; For square matrices, ${\rm Tr}(M)$, $\lambda_{\rm min}(M)$, and $\sigma_{\rm min}(M)$ represent the trace, minimal eigenvalue, and minimal singular value; $\mathrm{vec}(M)$ indicates the vectorized form; $\partial \mathbb{M}$ signifies the boundary of the set $\mathbb{M}$; $M \succ N$ ($M \succeq N$) implies that $M-N$ is positive definite (semidefinite); $\mathbb{S}^n_{+}$ ( $\mathbb{S}^n_{++}$) refers to the set of symmetric $n\times n$ positive semidefinite (definite) matrices; $\mathbb{N}$ stands for the set of natural numbers; $\mathbb{E}_{x}$ denotes taking expectation over $x$, and $I_n$ denotes the identity matrix.

\section{Problem Statement}
\label{sec:preliminary}
Consider the discrete-time linear time-invariant (LTI) dynamic model 
\begin{equation} 
\label{eq.statefunction}
\begin{aligned}
x_{t+1} &= Ax_t+Bu_t,\\
y_t &= Cx_t,
\end{aligned}
\end{equation}
with $x$ denoting the state, $y$ representing the output, and matrices $A\in \mathbb{R}^{n\times n}$, $B\in \mathbb{R}^{n\times m}$, and $C\in \mathbb{R}^{d\times n}$ describing the system dynamics. The linear quadratic regulator (LQR) problem aims to find a control policy to minimize the accumulated linear quadratic cost 
\begin{equation}   
\label{eq.objective}
\mathbb{E}_{x_0\sim \mathcal{D}}\Big[\sum_{t=0}^{\infty}(x_t^\top Qx_t+u_t^\top Ru_t) \Big],
\end{equation}
where it is assumed that $\mathbb{E}_{x_0\sim\mathcal{D}}[x_0x_0^\top]\succ 0$, and $Q\in \mathbb{S}_{++}^{n}$ and $R\in \mathbb{R}_{++}^{m}$ are performance weights. The assumption on $\mathbb{E}_{x_0\sim\mathcal{D}}[x_0x_0^\top]\succ 0$ is quite standard in learning-based control \cite{fazel2018global, bu2019lqr} and can be somehow informally thought as the persistent excitation condition in data-driven control. 

The static output feedback (SOF) is defined as 
\begin{equation}   
\label{eq.sop}
u_t=-Ky_t,
\end{equation} 
with $K\in \mathbb{R}^{m\times d}$. Substituting the SOF controller into the dynamic model~\eqref{eq.statefunction} yields 
\begin{equation}   
\label{eq.closed-loop-system}
x_{t+1} = \mathcal{A}_{K}x_t, 
\end{equation}
where $\mathcal{A}_{K}:=A-BKC$. We can further reformulate the linear quadratic cost \eqref{eq.objective} as
\begin{equation}   
\label{eq.objective_with_K}
J(K) = \mathbb{E}_{x_0\sim \mathcal{D}}\Big[\sum_{t=0}^{\infty}x_t^\top (Q+C^\top K^\top R K C )x_t \Big].
\end{equation}
This study assumes that a stabilizing controller is present. We refer to the set of all stabilizing control gain $K$ as the feasible set, i.e., 
\begin{equation}
\mathbb{K}:=\{K\in \mathbb{R}^{m\times d}:\rho(\mathcal{A}_{K})<1\}.
\end{equation}

For the LTI systems \eqref{eq.statefunction}, the value function of state $x$ takes the quadratic closed-loop form 
\begin{equation}
\label{eq.V_expand}
V_{K}(x_t): = x_t^\top P_{K}x_t,
\end{equation}
where $P_K\in \mathbb{S}_+^n$. 

We define the accumulated state correlation matrix as 
\begin{equation}
\label{eq.correlation_matrix}
\Sigma_K:=\mathbb{E}_{x_0\sim \mathcal{D}}\sum_{t=0}^{\infty}x_t x_t^\top.
\end{equation}
If the initial state correlation matrix is positive definite, i.e.,
\begin{equation}
X_0 :=  \mathbb{E}_{x_0\sim \mathcal{D}}[x_0x_0^\top] \succ 0,
\end{equation}
one has that the minimal singular value of $X_0$ 
\begin{equation}
\mu :=  \sigma_{\rm min}(X_0)\
> 0.
\end{equation}
Since $\Sigma_K \succeq X_0$, it is straightforward that
\begin{equation}
\sigma_{\rm min}(\Sigma_K)\ge \mu.
\end{equation}

With $P_K$ and $\Sigma_K$, it is well known in the literature \cite{fazel2018global,lee2018primal} that SOF control of discrete-time LTI systems with quadratic cost can be formulated into the following problem.
\begin{problem}[Policy Optimization for SOF]
\label{pro.SOF}
\begin{equation}   
\label{eq.cost_in_P}
\begin{aligned}
\min_{K\in \mathbb{K}}\  J(K)={\rm Tr}(P_KX_0)={\rm Tr}((Q+C^\top K^\top R K C)\Sigma_K),
\end{aligned}
\end{equation}
\textnormal{where $P_{K}$ and $\Sigma_{K}$ satisfy the following Lyapunov equations}
\begin{subequations}
\begin{align} 
\label{eq.lyapunov_equation}
&\begin{aligned}
P_K &= Q + C^\top K^\top RKC+\mathcal{A}_{K}^\top P_K\mathcal{A}_{K},
\end{aligned}\\
&\begin{aligned} 
\label{eq.lyapunov_equation_sigma}
\Sigma_K = X_0 + \mathcal{A}_{K} \Sigma_K\mathcal{A}_{K}^\top.
\end{aligned} 
\end{align} 
\end{subequations}
\end{problem}

The unique positive definite solution of \eqref{eq.lyapunov_equation} can be expressed as
\begin{equation}   
\label{eq.P_expand}
\begin{aligned}
P_K= \sum_{j=0}^{\infty}{\mathcal{A}_{K}^\top}^j(Q+C^\top K^\top RKC){\mathcal{A}_{K}}^j.
\end{aligned}
\end{equation} 
The formulation of Problem \ref{pro.SOF} enables us to derive the analytical policy gradients to analyze the optimization landscape.  For this problem, we make the following standard assumption.
\begin{assumption}
\label{assumption.control_observe}
$(A,B)$ is controllable, $(C, A)$ is observable, and $C$ has independent rows.
\end{assumption}

Note that the feasible set 
$\mathbb{K}$ of Problem \ref{pro.SOF} can possess a disconnected domain, replete with local minima, saddle points, or even local maxima at the stationary points of each component \cite{fatkhullin2020CTSOF,feng2020connectivity,bu2019topological}; 
therefore, Problem \ref{pro.SOF} is generally non-convex, making the convergence analysis far from straightforward.

\section{Gradients and Hessian}
\label{sec.gradient}
In this section, we give the analytical expression for both the gradient and Hessian. The derivations follow similar lines as the state-feedback LQR case \cite{fazel2018global,jansch2020Mjump}. 

\begin{lemma}[Policy Gradient Formula]
\label{lemma:gradient}
\textnormal{
 For any control gain $K$ in the feasible set $\mathbb{K}$, we have 
\begin{equation} 
\label{eq.gradient}
\nabla J(K) = 2E_K\Sigma_KC^\top,
\end{equation}
with $E_K := (R+B^\top P_K B)KC-B^\top P_KA$.}
\end{lemma}
\begin{proof}
From \eqref{eq.V_expand} and \eqref{eq.lyapunov_equation}, it follows that 
\begin{equation}
\begin{aligned}
V_K(x_0) =& x_0^\top (Q + C^\top K^\top RKC)x_0+x_0^\top\mathcal{A}_{K}^\top P_K\mathcal{A}_{K}x_0\\
=&x_0^\top (Q + C^\top K^\top RKC)x_0+V_K(\mathcal{A}_{K}x_0).
\end{aligned}
\end{equation}
Taking the gradient of $V_K(x_0)$ w.r.t. $K$, one has
\begin{equation}
\begin{aligned}
\nabla V_K(x_0) =& 2E_K x_0x_0^\top C^\top +x_1^\top \nabla P_K x_1\big|_{x_1=\mathcal{A}_{K}x_0}\\
=& 2E_K x_0x_0^\top C^\top+\nabla V_K(x_1)\big|_{x_1=\mathcal{A}_{K}x_0}\\
=& 2E_K \sum_{t=0}^{\infty}(x_tx_t^\top) C^\top .
\end{aligned}
\end{equation}
By taking expectation w.r.t. $\mathcal{D}$, the expression of policy gradient is obtained 
\begin{equation}  
\nabla J(K) = \mathbb{E}_{x_0\sim \mathcal{D}}\nabla V_K(x_0)=2E_K\Sigma_KC^\top, 
\end{equation}
which completes the proof.
\end{proof}

Note that the objective function $J(K)$ is twice differentiable. Thus, the analytical form of the Hessian of the objective function can be derived. To simplify our analysis without delving into tensors, we analyze the Hessian along a certain matrix $Z\in \mathbb{R}^{m\times d}$, whose expression is as follows 
\begin{equation}  
\label{eq.quadaratic_of_hessian}
\begin{aligned}
\nabla^2J(K)[Z,Z]:&=\frac{{\rm d}^2}{{\rm d}\lambda^2}\Big|_{\lambda=0}J(K+\lambda Z)\\
&={\rm Tr}\left(\frac{{\rm d}^2}{{\rm d}\lambda^2}\Big|_{\lambda=0}P_{K+\lambda Z}X_0\right).
\end{aligned}
\end{equation}
\begin{lemma}
\label{lamma:hessian}
\textnormal{
For any control gain $K$ in the feasible set $\mathbb{K}$, the Hessian of the objective function $J(K)$ along a certain matrix $Z\in \mathbb{R}^{m\times d}$ is 
\begin{equation}  
\label{eq.Hessian formular}
\begin{aligned}
\nabla^2J(K)[Z,Z]&={\rm Tr}(2(ZC)^\top (B^{\top} P_KB + R)ZC \Sigma_K)\\
&~~~ - {\rm Tr}(4(BZC)^\top P'_K[Z]\mathcal{A}_{K} \Sigma_K),
\end{aligned}
\end{equation}
where 
\begin{equation} 
\label{eq.derivative_P}
\begin{aligned}
P'_K&[Z]=\sum_{j=0}^{\infty}{\mathcal{A}_{K}^\top}^j(C^\top Z^\top E_K  + E_K^\top ZC){\mathcal{A}_{K}}^j.
\end{aligned}
\end{equation}}
\end{lemma}

\begin{proof}
Denote $P'_K[Z]:=\frac{{\rm d}}{{\rm d}\lambda}\Big|_{\lambda=0}P_{K+\lambda Z}$. From \eqref{eq.lyapunov_equation}, we have
\begin{equation} 
\label{eq.P_derivative}
\begin{aligned}
P'_K[Z]&=C^\top Z^\top E_K  + E_K^\top ZC+\mathcal{A}_{K}^\top P'_K[Z]\mathcal{A}_{K}\\
&= \sum_{j=0}^{\infty}{\mathcal{A}_{K}^\top}^j(C^\top Z^\top E_K  + E_K^\top ZC){\mathcal{A}_{K}}^j.
\end{aligned}
\end{equation}
Then, its second derivative $P''_K[Z]:=\frac{{\rm d}^2}{{\rm d}\lambda^2}\Big|_{\lambda=0}P_{K+\lambda Z}$ can be derived as 
\begin{equation}   
\label{eq.second_derivative_P}
\begin{aligned}
P''_K[Z]=S_1+\mathcal{A}_{K}^\top P''_K[Z]\mathcal{A}_{K}=\sum_{j=0}^{\infty}{\mathcal{A}_{K}^\top}^jS_1{\mathcal{A}_{K}}^j,
\end{aligned}
\end{equation}
where
\begin{equation}
\begin{aligned}
S_1:=&2\Big(C^\top  Z^\top (R+B^{\top} P_KB)ZC \\
&-(BZC)^\top P'_K[Z]\mathcal{A}_{K}-\mathcal{A}_{K}^\top P'_K[Z]BZC\Big).
\end{aligned}
\end{equation}
Furthermore, from \eqref{eq.quadaratic_of_hessian} and \eqref{eq.lyapunov_equation_sigma}, we can show that 
\begin{equation}  
\begin{aligned}
\nabla^2J(K)[Z,Z]&={\rm Tr}(\sum_{j=0}^{\infty}{\mathcal{A}_{K}^\top}^jS_1{\mathcal{A}_{K}}^jX_0)\\
&={\rm Tr}(S_1\sum_{j=0}^{\infty}{{\mathcal{A}_{K}}^jX_0\mathcal{A}_{K}^\top}^j)\\
&={\rm Tr}(S_1\Sigma_K)\\
&={\rm Tr}(2(ZC)^\top (B^{\top} P_KB + R)ZC \Sigma_K)\\
&~~~ - {\rm Tr}(4(BZC)^\top P'_K[Z]\mathcal{A}_{K} \Sigma_K).
\end{aligned}
\end{equation}
\end{proof}

\section{Cost Function Properties}
\label{sec.property}
Building upon the derived explicit formulas for the gradient and Hessian, we are now ready to discuss the optimization landscape for the SOF problem. This section develops some essential properties of the cost function, which will play an important role in the final convergence analysis. The intermediate lemmas required by the property analysis are provided in Appendix \ref{appen.intermediate lemma}.

\begin{lemma}[Coercive Property]
\label{lemma.coercivity}
\textnormal{The SOF cost \eqref{eq.cost_in_P} is coercive, that is, for all sequence $\{K_i\}_{i=1}^{\infty}\subseteq \mathbb{K}$, we have
\begin{equation}
\nonumber
J(K_i) \rightarrow +\infty,\quad {\rm if}\   K_i \rightarrow K \in \partial \mathbb{K}\   {\rm or}\  \|K_i\|\rightarrow +\infty.
\end{equation}}
\end{lemma}
See Appendix \ref{append:proof_coercivity} for detailed proof. Based on the coercivity nature, we can obtain the compactness of the sublevel set.
\begin{lemma}[Compactness of Sublevel Set]
\label{lemma.compact}
\textnormal{Given a scalar $\alpha \ge J(K^{\star})$ with the globally optimal SOF gain $K^{\star}$, the sublevel set $\mathbb{K}_{\alpha}:=\{K | J(K)\le \alpha\} \subseteq \mathbb{K}$ is compact.}

\end{lemma}

\begin{proof}
Upon the coercivity proven in Lemma \ref{lemma.coercivity}, and referring to \cite[Proposition 11.12]{bauschke2011convex}, it becomes evident that the set $\mathbb{K}_{\alpha}$ is bounded.  Given the continuity of $J(K)$ over $\mathbb{K}$, it follows that $\mathbb{K}_{\alpha}$ is also closed, which completes the proof.
\end{proof}
With the compactness property in place, it becomes possible to demonstrate that the monotonicity of the objective function guarantees that the line segment between two neighboring iterations remains within $\mathbb{K}_{\alpha}$.

\begin{lemma}[Smoothness on Sublevel Set]
\label{lemma.L-smooth}
\textnormal{For all control gain $K$ in the sublevel set $\mathbb{K}_{\alpha}$, the norm of the Hessian of the cost function is bounded by a constant, i.e., $\|\nabla^2 J(\mathrm{vec}(K))\|\le L$, where 
\begin{equation}
\nonumber
L=\frac{2\alpha}{\sigma_{\rm min}(Q)}\left(\|R\|+\frac{\alpha}{\mu}\left(1+\frac{2\zeta_1}{\|B\|\|C\|}\right)\|B\|^2\right)\|C\|^2,
\end{equation}
with 
\begin{equation}
\label{eq.zeta1}
\zeta_1 = \frac{1}{\sigma_{\rm min}(Q)}\left(\frac{\alpha}{\mu}\left(1+\|B\|^2\|C\|^2\right)+\|R\|\|C\|^2 \right)-1.
\end{equation}
}
\end{lemma}

\begin{proof}
From \eqref{eq.quadaratic_of_hessian}, applying the Taylor series expansion about direction $Z$, we can show that 
\begin{equation}
\nabla^2J(K)[Z,Z]=\mathrm{vec}(Z)^\top \nabla^2J(\rm vec(K)) \mathrm{vec}(Z).
\end{equation}
Since $\nabla^2 J(\mathrm{vec}(K))$ is symmetric, one has
\begin{equation}
\label{eq.supremum_of_hessian}
\begin{aligned}
\|\nabla^2 J(\mathrm{vec}(K))\|&=\sup_{\|Z\|_F=1}|\mathrm{vec}(Z)^\top \nabla^2J(\rm vec(K)) \mathrm{vec}(Z)|\\
&=\sup_{\|Z\|_F=1}|\nabla^2J(K)[Z,Z]|.
\end{aligned}
\end{equation}

Based on \eqref{eq.Hessian formular}, we further have 
\begin{equation}
\label{eq.upper_bound_of_hessian}
\begin{aligned}
\|\nabla^2 J&(\mathrm{vec}(K))\|\\
\le& 2\sup_{\|Z\|_F=1}|{\rm Tr}(C^\top Z^\top (R+B^{\top} P_KB)ZC \Sigma_K)|\\
&+4\sup_{\|Z\|_F=1}|{\rm Tr}((BZC)^\top P'_K[Z]\mathcal{A}_{K} \Sigma_K)|\\
=:&2q_1+4q_2.
\end{aligned}
\end{equation}
Actually, $q_1$ and $q_2$ are bounded above by
\begin{subequations}
\label{eq.upper_of_q1_q2}
\begin{align} 
&\begin{aligned}
\label{eq.upper_of_q1_final}
 q_1 \le \frac{J(K)}{\sigma_{\rm min}(Q)}\left(\|R\|+\frac{J(K)}{\mu}\|B\|^2\right)\|C\|^2,
 \end{aligned}\\
&\begin{aligned}
\label{eq.upper_of_q2_new_final}
q_2\le \frac{\zeta_{1}J(K)^2}{\mu\sigma_{\rm min}(Q)}\|B\|\|C\|.
\end{aligned}
\end{align} 
\end{subequations}
The detailed derivation of \eqref{eq.upper_of_q1_q2} is referred to Appendix \ref{appendix.bound}.

Plugging  \eqref{eq.upper_of_q1_q2} into \eqref{eq.upper_bound_of_hessian},  we finally complete the proof.
\end{proof}

In light of Lemma \ref{lemma.L-smooth}, consider any scalar $\delta\in [0,1]$ and any control gains $K$ and $K'$ residing in the sublevel set $\mathbb{K}_\alpha$. If any point along the segment defined by $(1 - \delta) K + \delta K'$ remains within this sublevel set, then the cost function has 
\begin{equation}
\label{eq.L-smooth}
\begin{aligned}
J(K')&\le J(K)+{\rm Tr}\left(\nabla J(K)^\top(K'-K)\right)+\frac{L}{2}\|K-K'\|_F^2.
\end{aligned}
\end{equation}
Moreover, if the cost function exhibits global $L$-smoothness, it is widely acknowledged that the gradient descent method can attain a stationary point with a gradient step complexity that is independent of the dimension \cite{polyak1963gradient,nesterov1998introductory}. However, the $L$-smooth property \eqref{eq.L-smooth} and its derived conclusions are not applicable to all control gains $K, K'\in \mathbb{K}_\alpha$ because the domain can be non-convex or even disconnected \cite{fatkhullin2020CTSOF}. 

Denote the output correlation matrix as 
\begin{equation}
\mathcal{L}_K:=C\Sigma_{K}C^\top=\mathbb{E}_{x_0\sim \mathcal{D}}\sum_{t=0}^{\infty}y_t y_t^\top.
\end{equation} 
Next, we will give the gradient domination condition for the fully observed case. These results are already established in the literature \cite{fazel2018global}; we provide a short proof in Appendix \ref{appendix.proof_dominance} for completeness.
\begin{lemma}[Gradient Domination]
\label{lemma.gradient_dominance}
\textnormal{Denote $\mathbb{C}:=\{C \in \mathbb{R}^{n\times n}: {\rm rank}(C)=n\}$. The globally optimal gain of the SOF problem and the globally optimal performance of the corresponding LQR problem with the state-feedback controller are denoted as $K^{\star}$ and $J_s^{\star}$, respectively. Assuming $X_0 \succ 0$ and that the control gain $K$ attains a finite performance, we can express the upper bound of the cost function for $K$ as 
\begin{equation}
\label{eq.gradient_dominance}
J(K)-J_s^{\star}\le \frac{\|\Sigma_{K^{\star}}\|\|\nabla J(K)\|_F^2}{4\mu^2\sigma_{\rm min}(C)^2\sigma_{\rm min}(R)}, \ \forall C\in \mathbb{C}.
\end{equation}
Additionally, we have the following lower bound 
\begin{equation}
\label{eq:lower_bound_2}
J(K)-J_s^{\star}\ge \frac{\mu{\rm Tr}\big(E_K^\top  E_K\big)}{\|R+B^\top P_K B\|}, \ \forall C \in \mathbb{C}.
\end{equation}
When $1\le{\rm rank}(C)<n$, one only has
\begin{equation}
\label{eq.gradient_dominance_1}
J(K)-J(K^{\star})\le \|\Sigma_{K^{\star}}\|{\rm Tr}\big( E_K^\top (R+B^\top P_K B)^{-1} E_K \big).
\end{equation}}
\end{lemma}

The concept of gradient dominance is crucial for achieving global convergence in gradient descent algorithms, as it signifies that no stationary points exist aside from the global minimum \cite{polyak1963gradient,lojasiewicz1963topological}. Nevertheless, when $C$ is not a full-rank square matrix, this property ceases to be valid (see \cite[Example 3.4]{fatkhullin2020CTSOF}), making it challenging to achieve results beyond convergence towards a stationary point. Such limitations on gradient dominance extend to dynamic output-feedback controllers as well, as the set of stabilizing controllers contains at most two disconnected components \cite{tang2021analysis,duan2023optimization}.

\begin{lemma}[$M$-Lipschitz Continuous Hessian]
\label{lemma.continue_Hessian}
\textnormal{
For any control gain $K$ in the sublevel set $\mathbb{K}_\alpha$, define $\gamma:=\max_{K\in \mathbb{K}_{\alpha}} \|\mathcal{A}_{K}\|$ and denote the upper bound of $\|KC\|$ as $\psi$, whose explicit form is given Lemma \ref{lemma.K_upper_bound}. Given any scalar $\delta\in [0,1]$ and any control gains $K$ and $K'$ in the sublevel set, if any point along the segment between these two gains, represented as $(1 - \delta) K + \delta K'$, remains within the sublevel set, then the Hessian of the cost function satisfies
\begin{equation}
\label{eq.M_Hessian}
\begin{aligned}
\|\nabla^2 J(\mathrm{vec}(K'))-\nabla^2 J(\mathrm{vec}(K))\|_F\le M \|K'-K\|_F,
\end{aligned}
\end{equation}
where 
\begin{equation}
\nonumber
M =     \frac{4\alpha^2\sqrt{md}}{\mu\sigma_{\rm min}(Q)}\left(\big(\zeta_1+\frac{\zeta_2}{2}\big)\|B\|\|C\|+\zeta_3+\frac{\zeta_4}{2}\right)\|B\|\|C\|,
\end{equation}
$\zeta_1$ is defined in \eqref{eq.zeta1}, and other intermediate parameters are 
\begin{equation}
\begin{aligned}
\nonumber
\zeta_2 &= \frac{2\|C\|}{\sigma_{\rm min}(Q)}\left(\frac{\alpha\gamma}{\mu}\|B\| + \psi\|R\|\right),\\
\zeta_3 &= \frac{2\|C\|}{\sigma_{\rm min}(Q)}\left(\frac{\alpha}{\mu}(\zeta_1\gamma + \zeta_2\gamma + \|B\|\|C\|)\|B\| + \|R\|\|C\|\right),\\
\zeta_4 &= \frac{2\|C\|}{\sigma_{\rm min}(Q)} \left(\frac{\alpha}{\mu}(\zeta_1\gamma + \|B\|\|C\|)\|B\| + \|R\|\|C\|\right).
\end{aligned}
\end{equation}}
\end{lemma}
\begin{proof}
Similar to \eqref{eq.supremum_of_hessian}, since $\nabla^2 J(\mathrm{vec}(K))-\nabla^2 J(\mathrm{vec}(K'))$ is symmetric, we have
\begin{equation}
\begin{aligned}
&\|\nabla^2 J(\mathrm{vec}(K))-\nabla^2 J(\mathrm{vec}(K'))\|\\
&=\sup_{\|Z\|_F=1}|\nabla^2J(K)[Z,Z]-\nabla^2J(K')[Z,Z]|.
\end{aligned}
\end{equation}
By \eqref{eq.quadaratic_of_hessian}, we define
\begin{equation}
g(\delta):=\nabla^2 J((1 - \delta) K + \delta K')[Z,Z],
\end{equation}
and denote 
$\Bar{K}:=K+\delta(K'-K)$, $\Delta K:= K'-K$. Then, from \eqref{eq.quadaratic_of_hessian}, one has
\begin{equation}  
\label{eq.g_prime}
\begin{aligned}
g'(\delta)&= {\rm Tr}(\frac{\partial^3}{\partial\lambda^2\partial \delta}\Big|_{\lambda=0}P_{K+\delta(K'-K)+\lambda Z}X_0).
\end{aligned}
\end{equation}
By the fundamental theorem of calculus, it follows that
\begin{equation}
\label{eq.hessian_diff_calculus}
\begin{aligned}
\|\nabla^2 J(\mathrm{vec}(K))-\nabla^2 J(\mathrm{vec}(K'))\|&=\sup_{\|Z\|_F=1}|g(0)-g(1)|\\
&=\sup_{\|Z\|_F=1}|\int_{0}^{1}g'(\delta){\rm d}\delta|\\
&\le\int_{0}^{1}\sup_{\|Z\|_F=1}|g'(\delta)|{\rm d}\delta.
\end{aligned}
\end{equation}
Based on \eqref{eq.second_derivative_P}, we can observe that 
\begin{equation}  
\begin{aligned}
&\frac{\partial^3}{\partial\lambda^2\partial \delta}\Big|_{\lambda=0}P_{K+\delta(K'-K)+\lambda Z} = \sum_{j=0}^{\infty}{\mathcal{A}_{\bar{K}}^\top}^jS_2{\mathcal{A}_{\bar{K}}}^j,
\end{aligned}
\end{equation}
where 
\begin{equation}   
\begin{aligned}
&S_2:=2C^\top Z^\top B^{\top} \frac{\partial P_{\bar{K}}}{\partial \delta}BZC + 2C^\top Z^\top B^{\top} P'_{\bar{K}}[Z] B\Delta K C\\
& \qquad + 2C^\top \Delta K^\top B^{\top} P'_{\bar{K}}[Z] BZ C-2(BZC)^\top\frac{\partial P'_{\bar{K}}[Z]}{\partial \delta}{\mathcal{A}_{\bar{K}}} \\
&\qquad-2\mathcal{A}_{\bar{K}}^\top\frac{\partial P'_{\bar{K}}[Z]}{\partial \delta}BZC -  (B\Delta KC)^\top P''_{\bar{K}}[Z]{\mathcal{A}_{\bar{K}}}\\
& \qquad -\mathcal{A}_{\bar{K}}^\top P''_{\bar{K}}[Z]B\Delta K C.
\end{aligned}
\end{equation}
According to \eqref{eq.g_prime}, it follows that 
\begin{equation}  
\begin{aligned}
g'(\delta)
&={\rm Tr}(2(BZC)^\top\frac{\partial P_{\bar{K}}}{\partial \delta}BZC \Sigma_{\bar {K}})\\
&~~~ +{\rm Tr}(4(BZC)^\top P'_{\bar{K}}[Z] B\Delta K C\Sigma_{\bar {K}})\\
&~~~ -{\rm Tr}(4(BZC)^\top\frac{\partial P'_{\bar{K}}[Z]}{\partial \delta}{\mathcal{A}_{\bar{K}}}\Sigma_{\bar {K}})\\
&~~~ -{\rm Tr}(2(B\Delta KC)^\top P''_{\bar{K}}[Z]{\mathcal{A}_{\bar{K}}}\Sigma_{\bar {K}}).
\end{aligned}
\end{equation}
Similar to the derivation of \eqref{eq.upper_of_q1_q2}, we can further show that 
\begin{equation}  
\begin{aligned}
\sup_{\|Z\|_F=1}|g'(\delta)|&\le2\|B\|^2\|C\|^2\|\frac{\partial P_{\bar{K}}}{\partial \delta}\|{\rm Tr}(\Sigma_{\bar {K}})\\
&\quad +4\|B\|^2\|C\|^2\|P'_{\bar{K}}[Z]\|{\rm Tr}(\Sigma_{\bar {K}})\|\Delta K\|\\
&\quad +4\|B\|\|C\|\|\frac{\partial P'_{\bar{K}}[Z]}{\partial \delta}\|{\rm Tr}(\Sigma_{\bar {K}})\\
&\quad +2\|B\|\|C\|\|P''_{\bar{K}}[Z]\|{\rm Tr}(\Sigma_{\bar {K}})\|\Delta K\|.
\end{aligned}
\end{equation}

According to Lemma \ref{lemma.L-smooth}, we know that $P'_{\bar K}[Z] \preceq \zeta_1 P_{\bar K}$. As a matter of fact, we can also show that $\frac{\partial P_{\bar{K}}}{\partial \delta} \preceq \zeta_2 \|\Delta K\|P_{\bar{K}}$, $\frac{\partial P'_{\bar{K}}[Z]}{\partial \delta} \preceq \zeta_3 \|\Delta K\|P_{\bar{K}}$,  and $P''_{\bar{K}}[Z]\preceq \zeta_4 P_{\bar{K}}$ (see Appendix \ref{appendix.continue_Hessian} for detailed derivations). Utilizing the results of Lemma \ref{lemma.upper_bound}, we can further show that 
\begin{equation}  
\label{eq.third_derivative}
\begin{aligned}
&\sup_{\|Z\|_F=1}|g'(\delta)|\\
&\le2\|B\|\|C\|\big((2\zeta_1+\zeta_2)\|B\|\|C\|+2\zeta_3+\zeta_4\big)\frac{\alpha^2\|\Delta K\|}{\mu\sigma_{\rm min}(Q)}.
\end{aligned}
\end{equation}
Plugging \eqref{eq.third_derivative} into \eqref{eq.hessian_diff_calculus} and remembering $\|X\|\le \|X\|_F \le \sqrt{{\rm rank}(X)}\|X\|$, we finally complete the proof.
\end{proof}

To the best of our knowledge, the Lipschitz continuity of the Hessian for the SOF cost function has not been previously examined. Nonetheless, this discovery is notable for enhancing the convergence towards a local minimum in non-convex optimization scenarios, under relatively mild conditions \cite{nesterov1998introductory}. Moreover, recent studies \cite{jin2017escape,ge2015escaping,carmon2018accelerated} indicate that the Hessian Lipschitz property facilitates efficient navigation away from strict saddle points in general gradient-based non-convex optimization problems.

\begin{remark}
\label{remark.lemma_difference}
\textnormal{The coercive property, compactness of the sublevel set, and $L$-smoothness of the cost function in the SOF problem, can be deemed as partially observed counterparts to the properties of the state-feedback LQR cost. The associated proofs follow similar lines as the state-feedback LQR case \cite{bu2019lqr,jansch2020Mjump}. Different from these properties, to the best of our knowledge, we are the first to establish the $M$-Lipschitz continuous Hessian in both SOF and state-feedback LQR problems. Notably, this property cannot be straightforwardly derived using methods akin to those employed for establishing $L$-smoothness \cite{jansch2020Mjump}. This is because the analysis of $\nabla^3 J(\mathrm{vec}(K))$ necessitates complicated tensor operations. To circumvent these tensor-related complexities, we directly establish the $M$-Lipschitz continuous Hessian by adhering to the Lipschitz continuity definition \eqref{eq.M_Hessian}.} 
\end{remark}

\section{Convergence}
\label{sec:convergence}
This section presents new convergence findings for three variants of policy gradient methods applied to SOF: the vanilla policy gradient, the natural policy gradient, and the Gauss-Newton method. These three methods have been extensively analyzed in related studies \cite{fazel2018global,jansch2020Mjump,bu2019lqr}. Given that the cost function $J(K)$ is non-convex, the properties outlined in Section \ref{sec.property} play a crucial role in facilitating the convergence analysis for these policy gradient methods.

\subsection{Vanilla Policy Gradient}
The vanilla policy gradient method iterates as follows
\begin{equation}
\label{eq.pg}
K_{i+1}=K_i-\eta \nabla J(K_i),
\end{equation}
with the initial gain $K_0 \in \mathbb{K}$ and  the step size $\eta$. If the line segment $[K_i, K_{i+1}]$ is verified to lie within a sublevel set, then the convergence of the iteration gain obtained by \eqref{eq.pg} can be directly inferred by leveraging the $L$-smoothness property specified in Lemma \ref{lemma.L-smooth} \cite{nesterov1998introductory}. Before we proceed to the main proofs, let us give the following definition.

\begin{definition}
\textnormal{Given a differentiable function $J(\cdot)$, if $\|\nabla J(K)\|_F\le \epsilon$, $K$ is an $\epsilon$-stationary point.} 
\end{definition}

\begin{theorem}
\label{theorem: first_order_convergence}
\textnormal{
Assume that $J(K_0)=\alpha$ and $X_0\succ 0$. If we run the vanilla policy gradient method \eqref{eq.pg} with any step size $\eta \in (0,\:1/L]$, $J(K_i)$ is monotonically diminishing (which indicates $K_{i}\in \mathbb{K}_\alpha \subseteq \mathbb{K}$, i.e., $K_i$ is stabilizing), and an $\epsilon$-stationary point will be obtained in  
\begin{equation}
\label{eq.nearly_converge_rate}
\frac{2\alpha}{\eta\epsilon^2}
\end{equation}
iterations. Additionally, for each iteration $i$, the line segment $[K_i,K_{i+1}] \subseteq \mathbb{K}_\alpha$. If $C$ is full rank, an $\epsilon_J$-optimal gain $K_N$, i.e., $J(K_N)-J_s^{\star} \le \epsilon_J$, is obtained when the iteration step  
\begin{equation}
\label{eq.converge_rate}
N\ge \frac{\|\Sigma_{K^{\star}}\|}{2\eta\mu^2\sigma_{\rm min}(C)^2\sigma_{\rm min}(R)}\log \left(\frac{J(K_0)-J_s^{\star}}{\epsilon_J}\right).
\end{equation}}
\end{theorem}
\begin{proof}
We first define an open set $\mathbb{K}_{\alpha}^o:=\{K | J(K)< \alpha\} \subseteq \mathbb{K}$, whose complement $(\mathbb{K}_{\alpha}^o)^c$ is a closed set. By invoking Lemma~\ref{lemma.L-smooth}, for a given $\phi\in(0,1)$, there is a positive number $\varsigma$ so that $\|\nabla^2J(\mathrm{vec}(K_i))\| \le L < L + \phi L$ holds for all $K_i \in \mathbb{K}_{\alpha} \subset \mathbb{K}_{\alpha+\varsigma}$. 

Due to the compactness of $\mathbb{K}_{\alpha}$ established by Lemma \ref{lemma.compact}, the distance between $\mathbb{K}_{\alpha}$ and $(\mathbb{K}_{\alpha+\varsigma}^o)^c$, represented by ${d} = \inf\{\|K_i - K_j\|, \forall K_i \in \mathbb{K}_{\alpha}, \forall K_j \in (\mathbb{K}_{\alpha+\varsigma}^o)^c\}$, is guaranteed to be positive. Now, choose a step size $t$ so that ${t} \le \min \{2/(L + \phi L), {d}/\|\nabla J(K_i)\|\}$. This ensures that the segment $[K_i, K_i - {t} \nabla J(K_i)] \subseteq \mathbb{K}_{\alpha+\varsigma}$. According to the L-smoothness result~\eqref{eq.L-smooth}, one has 
\begin{equation}
\label{eq.eq50}
\begin{aligned}
J(K_i) \ge J(K_i\!-\!{t} \nabla J(K_i)) + {t}\!\left(\!1\!-\!\frac{(L\!+\!\phi L){t}}{2}\!\right)\!\|\nabla J(K_i)\|_F^2.
\end{aligned}
\end{equation}

Given the range of step size $t$, we confirm $J(K_i -{t} \nabla J(K_i)) \le J(K_i) < \alpha$, ensuring the iteration point $K_i -{t} \nabla J(K_i) \in \mathbb{K}_\alpha$ and the segment $[K_i, K_i -{t} \nabla J(K_i)] \subseteq \mathbb{K}_{\alpha}$.  By applying similar reasoning through \eqref{eq.L-smooth}, we can demonstrate that $[K,K-2t \nabla J(K)]\in\mathbb{K}_\alpha$ when $2t \le 2/((1+\phi)L)$. Furthermore, using induction, we generalize this result for  $T\in \mathbb{N}^+$ steps, establishing that $[K,K-Tt \nabla J(K)]\in\mathbb{K}_\alpha$ if $Tt \le 2/((1+\phi)L)$. 

Next, we consider a step size $\eta \le 1/L$. We can then choose a positive $t>0$ and a positive integer $T$ so that  $T{t} \in [\eta,\:2/(L + \phi L)]$. Then, the segment $[K_i, K_i - \eta \nabla J(K_i)] \subseteq \mathbb{K}_\alpha$. Following a parallel argument to that for \eqref{eq.eq50}, we get
\begin{equation}
\label{eq:monotony_gd}
\begin{aligned}
J(K_i - \eta \nabla J(K_i)) \le J(K_i)-\frac{\eta}{2}\|\nabla J(K_i)\|_F^2,
\end{aligned}
\end{equation}
where the inequality takes into account that $\eta \le 1/L$, with the boundary $1/L$ selected to achieve the fastest convergence rate. 

Given $J(K_0)=\alpha$, \eqref{eq:monotony_gd} indicates that $K_1 \in \mathbb{K}_\alpha$. Then, for any iteration $i$, we can use mathematical induction to arrive at
\begin{equation}
\label{eq:monotony}
J(K_{i+1})\le J(K_i)-\frac{\eta}{2}\|\nabla J(K_i)\|_F^2.
\end{equation}
Also, the line segment $[K_i, K_{i+1}] \subseteq \mathbb{K}_\alpha$. Summing up the above inequality yields
\begin{equation}
\frac{\eta}{2}\sum_{i=0}^{N}\|\nabla J(K_i)\|_F^2 \le J(K_0) - J(K_{N+1})\le J(K_0) - J(K^{\star}).
\end{equation}
It then follows that  $\lim_{i \rightarrow \infty} \|\nabla J(K_i)\|_F^2 = 0$ and 
\begin{equation}
\min_{0\le i\le N} \|\nabla J(K_i)\|_F^2 \le  \frac{2(J(K_0) - J(K^{\star}))}{\eta N} \le \frac{2\alpha}{\eta N},
\end{equation}
which shows the vanilla policy gradient can reach an $\epsilon$-stationary point within $\frac{2\alpha}{\eta\epsilon^2}$ iterations.

Furthermore, when $C\in\mathbb{C}$, combining \eqref{eq:monotony} and \eqref{eq.gradient_dominance} yields
\begin{equation}
J(K_{i+1})-J(K_{i}) \le -\frac{2\eta\mu^2\sigma_{\rm min}(C)^2\sigma_{\rm min}(R)}{\|\Sigma_{K^{\star}}\|}(J(K_i)-J_s^{\star}).
\end{equation}
This subsequently results in
\begin{equation}
J(K_{i})-J_s^{\star} \le \left(1-\frac{2\eta\mu^2\sigma_{\rm min}(C)^2\sigma_{\rm min}(R)}{\|\Sigma_{K^{\star}}\|}\right)^i(J(K_0)-J_s^{\star}).
\end{equation}
This proves the second claim of this theorem.
\end{proof}

Theorem \ref{theorem: first_order_convergence} establishes that, starting with an initial stabilizing control gain, the vanilla policy gradient method for the SOF problem ensures both the recursive stability of the control policy and a monotonically decreasing cost function. Moreover, the convergence rate to a stationary point is dimension-independent. To offer a unified view that encompasses both SOF and state-feedback LQR, our findings also reveal that the vanilla policy gradient method globally converges to a unique minimum at a linear rate when the state is fully observed. In this context, the convergence rate outlined in \eqref{eq.converge_rate} aligns with the conclusions in \cite[Theorem 7]{fazel2018global}. Notably, in contrast to \cite[Theorem 7]{fazel2018global}, we provide an explicit upper bound of the step size $\eta$ such that \eqref{eq:monotony} is satisfied. 

Although the convergence to stationary points of the vanilla policy gradient for SOF has been established, it is important to note that these stationary points can be local minima, saddle points, or even local maxima. Next, we will proceed to demonstrate that under mild assumptions, the vanilla method can indeed converge to a local minimum.
\begin{theorem}
\label{theorem: linear_local_minima}
\textnormal{
Suppose all the conditions in Lemma~\ref{lemma.L-smooth} and Lemma~\ref{lemma.continue_Hessian} hold. Assume that $\mathbb{K}_\beta \subset \mathbb{K}_\alpha$, where $\beta < \alpha$. So, the distance ${d} = \inf\{\|K_i - K_j\|, \forall K_i \in \mathbb{K}_{\beta}, \forall K_j \in (\mathbb{K}_{\alpha}^o)^c\}$ between two compact sets is positive. For the sublevel set $\mathbb{K}_{\beta}$, assume that there is a local minimum $K^\#$ with $l= \lambda_{\rm min}(\nabla^2 J(\mathrm{vec}(K^\#))>0$. Given that the initial gain $K_0$ is sufficiently close to this local minimum $K^\#$, denoted by an initial error  $r_0=\|K_0-K^\#\|_F < \bar{r} =2l/M$, and fulfilling the condition $\bar{r}r_0/(\bar{r}-r_0)\le {d}$, the vanilla policy gradient using a constant step size $\eta \le 1/L$ has an upper error bound:
\begin{equation}
\|K_i-K^{\#}\|_F\le \frac{\bar{r}r_0}{\bar{r}-r_0}\big(\frac{1}{1+\eta l}\big)^i.
\end{equation}
}
\end{theorem}
\begin{proof}
Denote the set of gains around the minimum $K^{\#}$ as $\mathbb{K}^{\#}:=\{K\in \mathbb{R}^{m\times d}:\|K-K^{\#}\|_F\le \bar{r}r_0/(\bar{r}-r_0) \le {d}\}$. Then, we have $\mathbb{K}^{\#} \subset \mathbb{K}_{\alpha}$. For any scalar $\delta\in [0,1]$ and any control gains $K,K'\in \mathbb{K}^{\#}$, it follows that $(1 - \delta) K + \delta K'\in \mathbb{K}^{\#} \subset \mathbb{K}_{\alpha}$. Therefore, the conclusions of Lemma~\ref{lemma.L-smooth} and Lemma~\ref{lemma.continue_Hessian} can be applied directly. Finally, the upper error bound of iterative gain can be immediately derived from~\cite[Theorem 1.2.4]{nesterov1998introductory}. 
\end{proof}

When initialized near local minima, Theorem~\ref{theorem: linear_local_minima} assures that vanilla policy gradient will exhibit linear convergence concerning the control gain. Although the aforementioned theoretical analysis relies on full awareness of model parameters and cost function details, it is worth noting the applicability of this analysis in model-free environments. In such settings, data-driven approaches like zeroth-order optimization techniques can be employed to offer an unbiased estimation of $\nabla J(K)$~\cite{conn2009introduction,nesterov2017random, fazel2018global}. Hence, our findings suggest that data-driven methods can also effectively handle discrete-time SOF problems, provided the gradient is approximated with reasonable precision.

\subsection{Natural Policy Gradient}
Besides the vanilla policy gradient method, the natural policy gradient method is also widely used in RL research \cite{fazel2018global,jansch2020Mjump,kakade2001natural}.  The natural gradient method iterates as follows 
\begin{equation}
\label{eq.natural_pg}
K'=K-\eta \nabla^{\rm NA} J(K) ,
\end{equation}
where 
\begin{equation}
\nonumber
\nabla^{\rm NA} J(K)=\nabla J(K) \mathcal{L}_K^{-1}
\end{equation}
is the natural policy gradient. More explanations for this update rule can be found in \cite{fazel2018global}.

\begin{theorem}
\label{theorem: natural_pg}
\textnormal{
Suppose $J(K_0)=\alpha$ and $X_0 \succ 0$. The cost $J(K_i)$ of natural gradient descent \eqref{eq.natural_pg} is monotonically diminishing (which indicates $K_{i}\in \mathbb{K}_\alpha \subseteq \mathbb{K}$, i.e., $K_i$ is stabilizing), and an $\epsilon$-stationary point, i.e., $ \|\nabla^{\rm NA} J(K_i) \|_F \le \epsilon$, can be reached in
\begin{equation}
\label{eq.nearly_convergence_rate_npg}
\frac{2\alpha}{\eta\mu\sigma_{\rm min}(C)^2 \epsilon^2}
\end{equation}
iterations, where the step size $\eta \le \mu \sigma_{\rm min}(C)^2/L$. 
If $C$ is full rank, an $\epsilon_J$-optimal control gain $K_N$, satisfying $J(K_N)-J_s^{\star} \le \epsilon_J$, is achieved when the iteration step 
\begin{equation}
\label{eq.convergence_rate_npg}
N\ge \frac{\|\Sigma_{K^{\star}}\|}{2\eta\mu\sigma_{\rm min}(R)}\log \left(\frac{J(K_0)-J_s^{\star}}{\epsilon_J}\right).
\end{equation}}
\end{theorem}

The proof of Theorem \ref{theorem: natural_pg} is provided in Appendix \ref{appendix:natural_pg}, which is similar to that of Theorem \ref{theorem: first_order_convergence}. Theorem \ref{theorem: natural_pg} illustrates that the natural policy gradient technique also converges to a stationary point in SOF problems at a nearly dimension-free rate. The term ``nearly dimension-free rate" suggests that the convergence does not explicitly depend on the system dimension. Besides, the explicit form of the convergence rate  \eqref{eq.convergence_rate_npg} for the fully observed case ($C\in\mathbb{C}$) is also provided for completeness, which is consistent with the result given in \cite[Theorem 7]{fazel2018global}. Similar to the vanilla policy gradient method, the natural policy gradient method can also be implemented in a model-free manner. Since $\mathcal{L}_K=\mathbb{E}_{x_0\sim \mathcal{D}}\sum_{t=0}^{\infty}y_t y_t^\top$, one can just estimate $\nabla^{\rm NA} J(K)$ from  cost and output information trajectories. The numerical evidence given in existing studies \cite{fazel2018global,jansch2020Mjump} shows that the natural policy gradient method usually leads to a faster convergence speed than the vanilla policy gradient method. 

\subsection{Gauss-Newton Policy Gradient}
Next, we consider the Gauss-Newton policy gradient method, which iterates as follows 
\begin{equation}
\label{eq.gauss_newton_pg}
K'=K-\eta \nabla^{GN} J(K),
\end{equation}
where 
\begin{equation}
\nonumber
\nabla^{GN} J(K)=(R+B^\top P_K B)^{-1}\nabla J(K) \mathcal{L}_K^{-1}
\end{equation}
is the Gauss-Newton policy gradient. More explanations for this update rule can be found in \cite{fazel2018global}.

\begin{theorem}
\label{theorem: GN_convergence}
\textnormal{Suppose $J(K_0)=\alpha$ and $X_0 \succ 0$. If we run Gauss-Newton natural gradient descent \eqref{eq.gauss_newton_pg} with any step size $\eta \le \mu\sigma_{\rm min}(R)\sigma_{\rm min}(C)^2/L$, $J(K_i)$ is monotonically diminishing (which indicates $K_{i}\in \mathbb{K}_\alpha \subseteq \mathbb{K}$, i.e., $K_i$ is stabilizing), and an $\epsilon$-stationary point, i.e., $ \|\nabla^{\rm GN} J(K_i) \|_F \le \epsilon$, will be reached in 
\begin{equation}
\label{eq.nearly_convergence_rate_gpg}
\frac{2\alpha}{\eta \mu\sigma_{\rm min}(R)\sigma_{\rm min}(C)^2\epsilon^2}
\end{equation}
iterations. 
If $C\in \mathbb{C}$, an $\epsilon_J$-optimal control gain $K_N$, satisfying $J(K_N)-J_s^{\star} \le \epsilon_J$, is achieved when the iteration step 
\begin{equation}
\label{eq.convergence_rate_gpg}
N\ge \frac{\|\Sigma_{K^{\star}}\|}{2\eta\mu}\log \left(\frac{J(K_0)-J_s^{\star}}{\epsilon_J}\right).
\end{equation}}
\end{theorem}

See Appendix \ref{appendix:proof_GN} for details on deriving Theorem \ref{theorem: GN_convergence}. Theorem~\ref{theorem: GN_convergence} establishes the result of nearly dimension-free convergence to stationary points of the Gauss-Newton method. The explicit form of the convergence rate \eqref{eq.convergence_rate_gpg} for the fully observed case ($C\in\mathbb{C}$) is consistent with the result given in \cite[Theorem 7]{fazel2018global}. Different from the vanilla policy gradient and the natural policy gradient methods, the Gauss-Newton method is not suitable for model-free settings since it requires the knowledge of matrices $B$ and $P_K$. 

\begin{remark}
\textnormal{The Gauss-Newton method and the natural policy gradient method generally converge faster than the vanilla policy gradient method in terms of iteration number \cite{fazel2018global,jansch2020Mjump}. As a trade-off, these two methods need more information to calculate the update gradients, taking up more computational resources. Notably, the Gauss-Newton method is ill-suited for model-free settings, as its gradient estimation relies on matrices $B$ and $P_K$.} 
\end{remark}

\subsection{Impact of Initial Distribution}

Up to this point, we have demonstrated that all three policy gradient methods are capable of converging to stationary points at a nearly dimension-free rate. When implementing these policy gradient algorithms in practice, it is crucial to recognize that these stationary points are not fixed; they are influenced by the initial state distribution. 
\begin{proposition}
\label{Proposition:distribution}
\textnormal{Let $K^{\ddagger}$ represent the stationary point of the SOF problem. When $C$ lacks full rank and $K^{\ddagger}C \neq K^{\star}_s$,  $K^{\ddagger}$ is influenced by the initial state distribution.  Here, $K^{\star}_s$ represents the optimal solution for state feedback LQR.} 
\end{proposition}
\begin{proof}
For the case where $C \in \mathbb{C}$,  Lemma~\ref{lemma:gradient} and \eqref{eq.gradient_dominance} in Lemma~\ref{lemma.gradient_dominance} establish that 
\begin{equation}
\label{eq.E_K=0}
\|E_{K^{\star}_s}\|_F=0.
\end{equation}

From the definition of the SOF controller \eqref{eq.sop}, one has $$u_t=-KCx_t,$$
where $KC$ effectively serves as a state-feedback gain.

Lemma~\ref{lemma.gradient_dominance} asserts that  $K_s^{\star}$ is unique, which means $$\|E_{K}\|_F=0 \;\iff\; KC=K_s^{\star}.$$
However, when $C$ lacks full rank, it is possible that no gain $K$ will satisfy $KC=K_s^{\star}$.

If $K\in \mathbb{K}$ and $X_0\succ 0$, theory \cite[Lemma~2]{lee2018primal} suggests that the Lyapunov equation~\eqref{eq.lyapunov_equation_sigma} admits a unique positive definite solution $\Sigma_{K}$. For different initial distributions $\mathcal{D}$ and $\mathcal{D}'$, \eqref{eq.lyapunov_equation_sigma} indicates that 
$$\Sigma_{K^{\ddagger}}\neq \Sigma_{K^{\ddagger}}' \;{\rm if}\; X_0 \neq X_0' =\mathbb{E}_{x_0\sim \mathcal{D}'}x_0x_0^\top.$$

According to \eqref{eq.gradient}, a stationary point $K^{\ddagger}$ meets the condition 
$$\|\nabla J(K^{\ddagger})\|_F = 2\|E_{K^{\ddagger}}\Sigma_{K^{\ddagger}}C^\top\|_F=0.$$ 
However, if $\|E_{K^{\ddagger}}\|_F\neq 0$ (that is, $K^{\ddagger}C\neq K_s^\star$), we cannot guarantee that $\|\nabla J(K^{\ddagger})\|_F=\|E_{K^{\dagger}}\Sigma_{K^{\ddagger}}'C^\top\|_F$ will be zero for all possible distribution $\mathcal{D}'$, due to its influence on $\Sigma_{K^{\ddagger}}'$. 

Nevertheless,  when $C \notin \mathbb{C}$, the stationary point $K^{\ddagger}$ in SOF is influenced by the initial state distribution $\mathcal{D}$. In other words, different initial distributions could yield distinct stationary points unless $K^{\ddagger}C= K_s^\star$. 
\end{proof}

The foregoing theoretical discussion suggests that to achieve an effective SOF policy, the initial state distribution should be carefully selected to match the practical application conditions.

\section{Numerical Results}
\label{sec:simulation}
In this section, we will present some numerical simulations to verify the performance of the above gradient descent methods in optimizing SOF problems. Since the vanilla policy gradient method and the natural policy gradient method can be implemented in a model-free manner, their model-free versions are also developed and tested.

\subsection{Model-free Optimization}
In the model-free setting, the model parameters, $A$, $B$, $C$, $Q$, $R$, are unknown. In keeping with other work in the literature \cite{fazel2018global}, we assume the algorithm has access to the observation $y_t$ and running cost $c_t$ at each time step, where $c_t:=x_t^\top Qx_t+u_t^\top Ru_t$. Using the zeroth-order optimization approach \cite{conn2009introduction,nesterov2017random,fazel2018global}, Algorithm \ref{alg:model-free} provides a data-driven procedure to estimate the gradients of both vanilla and natural policy gradient methods. 

\begin{algorithm}[!htb]
\caption{Model-Free Vanilla and Natural Policy Gradient}
\label{alg:model-free}
\begin{algorithmic}
   \STATE Input: stabilizing policy gain $K_0$, number of trajectories $z$, roll-out length $l$, perturbation amplitude $r$, step size $\eta$
   \REPEAT
   \STATE \textbf{Gradient Estimation:}
   \FOR{$i=1,\cdots, z$}
   \STATE Sample $x_0$ from $\mathcal{D}$
   \STATE Simulate $K_{j}$ for $l$ steps starting from $x_0$ and observe $y_0, \cdots, y_{l-1}$ and $c_0, \cdots, c_{l-1}$. 
   \STATE Draw $U_i$ uniformly at random over matrices such that $\|U_i\|_F=1$, and generate a policy $K_{j,U_i}=K_j+rU_i$.
   \STATE Simulate $K_{j,U_i}$ for $l$ steps starting from $x_0$ and observe $c_0', \cdots, c_{l-1}'$.
   \STATE Calculate empirical estimates:
   \begin{equation}
   \nonumber
   \widehat{J_{K_{j}}^i}=\sum_{t=0}^{l-1} c_t,\; \widehat{\mathcal{L}_{K_{j}}^i}=\sum_{t=0}^{l-1} y_ty_t^\top,\; \widehat{J_{K_{j,U_i}}}=\sum_{t=0}^{l-1} c_t'.
   \end{equation}
   \ENDFOR
\STATE Return estimates:
   \begin{equation}
   \nonumber
    \widehat{\nabla J(K_j)} = \frac{1}{z}\sum_{i=1}^z\frac{\widehat{J_{K_{j,U_i}}}-\widehat{J_{K_{j}}^i}}{r}U_i,\; \widehat{\mathcal{L}_{K_j}}=\frac{1}{z}\sum_{i=1}^z\widehat{\mathcal{L}_{K_{j}}^i}.
   \end{equation}
    \STATE \textbf{Policy Update:}
    \STATE Vanilla policy gradient $K_{j+1}=K_j-\eta \widehat{\nabla J(K_j)}.$
    \STATE Natural policy gradient $K_{j+1}=K_j-\eta \widehat{\nabla J(K_j)}\widehat{\mathcal{L}_{K_j}}^{-1}.$
    \STATE $j=j+1$.
\UNTIL $\|\widehat{\nabla J(K_{j-1})}\|_F\le \epsilon$
\end{algorithmic}
\end{algorithm}

\vspace{-2em}

\subsection{Example I: Open-loop Unstable Linear System}
\label{sec.2-d example}

Consider an internally unstable linear system 
\begin{equation}
\label{eq.unstable_system}
    A = \begin{bmatrix}
    1.1 & 0.1 \\
    0 & 1.1
    \end{bmatrix}, 
    B = \begin{bmatrix}
    0 \\
    0.1
    \end{bmatrix}, 
    C = \begin{bmatrix}
    1.0 & 1.0
    \end{bmatrix},
\end{equation}
which is a discrete version of the famous Doyle's LQG example. Let $Q = 0.25 I_2$, $R = 0.2$, and $X_0 = 0.1 I_2$. We employ all three policy gradient methods in model-based settings and Algorithm \ref{alg:model-free} in model-free settings to learn a suboptimal SOF policy. The initial controller is set as $K_0 = 9$. The optimal gain $K^{\star}=4.0637$ can be found by solving several Lyapunov equations given in \cite[Theorem 1]{Yu2020AnED}. The step size of all methods is set as $\eta=0.2$. Besides, other hyperparameters of Algorithm \ref{alg:model-free} are set as: $r = 0.001$, $z = 2^{14}$, and $l = 100$.\footnote{Our code is available at \url{https://github.com/jieli18/sof}}

The relative errors of both the control gain and the cost function are presented in Fig. \ref{fig:eg1_results}, which are computed as $\|K-K^{\star}\|_F/\|K^{\star}\|_F$ and $|J(K)-J(K^{\star})|/|J(K^{\star})|$, respectively. We can easily observe that all model-based policy gradient methods converge to the optimal solution within 100 iterations. As expected, the two model-free methods, especially the model-free natural policy gradient method, converge more slowly and unsteadily than their model-based counterparts due to gradient estimation errors. These results provide numerical evidence for our theoretical convergence analysis.

\begin{figure}[!htbp]
\centering
\captionsetup{singlelinecheck = false,labelsep=period, font=small}
\captionsetup[subfigure]{justification=centering}
\subfloat[policy error ]{ \includegraphics[width=0.46\linewidth]{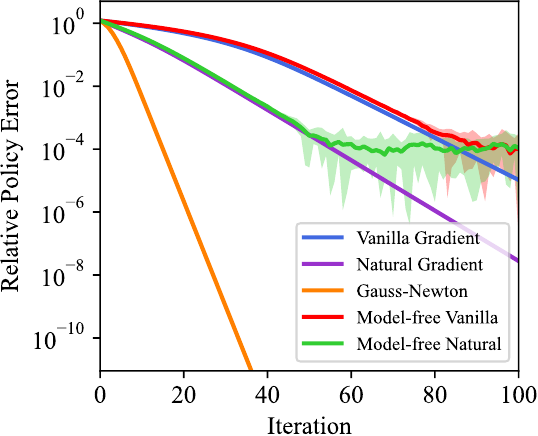}}
\subfloat[cost error]{\includegraphics[width=0.46\linewidth]{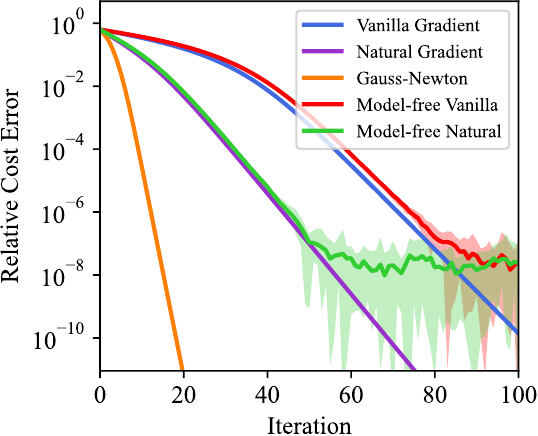}}
\caption{Learning curves of different methods for Example I. The solid lines correspond to the mean and the shaded regions correspond to an interval between maximum and minimum values over 10 runs.}
\label{fig:eg1_results}
\end{figure}

For the internally unstable system, such as \eqref{eq.unstable_system}, the stability of the controller can be assessed by evaluating the spectral radius of the closed-loop system matrix. This process, however, requires the knowledge of model dynamics. As a result, finding stabilizing controllers can be relatively complex when using model-free methods. In such instances, the trial and error approach could provide a practical strategy for obtaining an initial stabilizing controller. In terms of applying a controller to the dynamic system, the convergence or divergence of the observation output provides a useful criterion for determining the stability of the closed-loop system. Through the application of this manner, we are able to establish the set of stabilizing controllers for the internally unstable system~\eqref{eq.unstable_system}, which is $\mathbb{K} = \{K: K \in (2.1, 22.05)\}$.

We run all three policy gradient methods with 10 randomly generated initial stabilizing controllers. The relative errors of control gains are shown in Fig.~\ref{fig:initial_results}, where the curves of the same color start from the same initial point. It can be seen that all methods converge within 100 iterations under different initial controllers. This further confirms our theoretical convergence results within the context of an internally unstable system with randomly chosen initial controllers.

\begin{figure}[!htbp]
\centering
\captionsetup{singlelinecheck = false,labelsep=period, font=small}
\captionsetup[subfigure]{justification=centering}
\subfloat[vanilla gradient]{\includegraphics[width=0.46\linewidth]{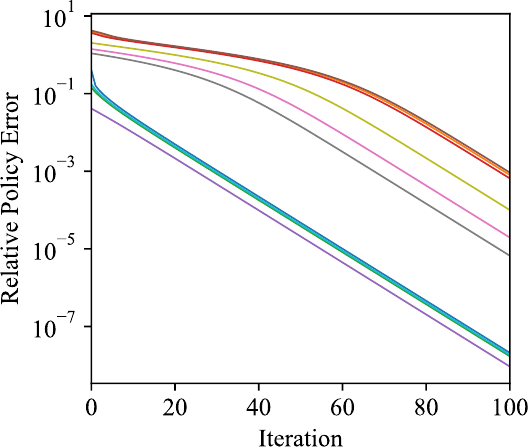}}
\subfloat[natural gradient]{\includegraphics[width=0.46\linewidth]{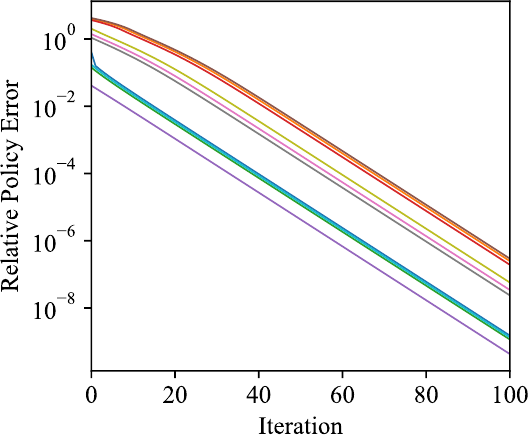}}
\\
\vspace{-1em}
\subfloat[Gauss-Newton]{\includegraphics[width=0.46\linewidth]{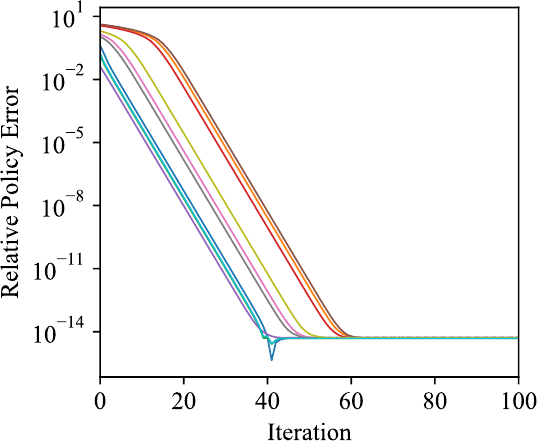}}
\subfloat[model-free vanilla]{\includegraphics[width=0.46\linewidth]{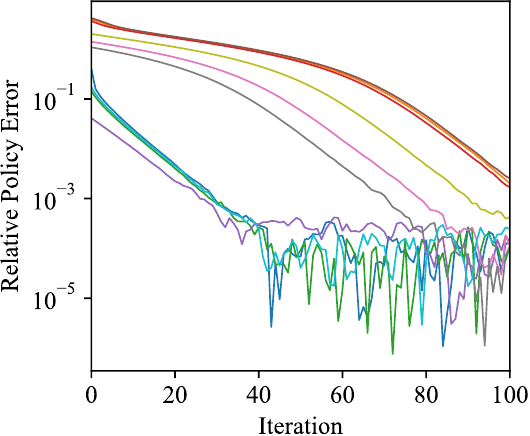}}
\\
\subfloat[model-free natural]{\includegraphics[width=0.46\linewidth]{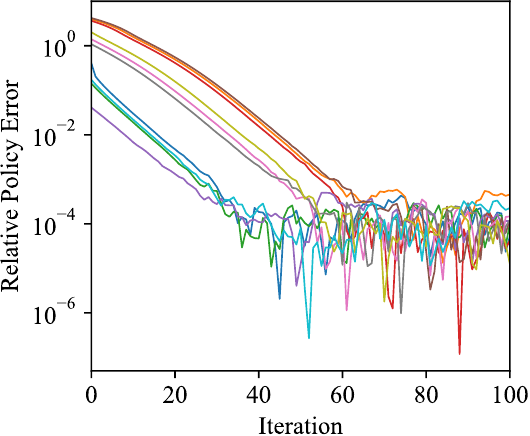}}
\caption{Learning curves of different methods with 10 different random initializations (corresponding to curves with different colors).}
\label{fig:initial_results}
\end{figure}

\subsection{Example II: Four-dimensional Open-loop Stable System}
Consider a circuit system given in \cite{lewis1992applied} 
with
\begin{equation}
    \nonumber
    \begin{aligned}
    A = \begin{bmatrix}
    0.90031 & -0.00015 & 0.09048 & -0.00452 \\
    -0.00015 & 0.90031 & 0.00452 & -0.09048 \\
    -0.09048 & -0.00452 & 0.90483 & -0.09033 \\
    0.00452 & 0.09048 & -0.09033 & 0.90483
    \end{bmatrix}, \\
    B = \begin{bmatrix}
    0.00468 & -0.00015 \\
    0.00015 & -0.00468 \\
    0.09516 & -0.00467 \\
    -0.00467 & 0.09516
    \end{bmatrix}, 
    C = \begin{bmatrix}
    1 & 1 & 0 & 0 \\
    0 & 1 & 0 & 0
    \end{bmatrix},
    \end{aligned}
\end{equation}
where $Q = {\rm diag}([0.1, 0.2, 0, 0])$, $R = {\rm diag}([10^{-6}, 10^{-4}])$, and $X_0 = I_4$. According to~\cite[Theorem 1]{Yu2020AnED}
, the optimal gain is
\begin{equation}
    \nonumber
    K^* = \begin{bmatrix}
    2.9738 & -7.2907 \\
    2.1067 & -12.5384
    \end{bmatrix}.
\end{equation}

We set $K_0 = \begin{bmatrix}
    0 & -1 \\
    0 & -2
    \end{bmatrix}$ for all methods and adopt the same hyperparameters as outlined in Section \ref{sec.2-d example}.
The relative errors in control gain and cost function for various methods are shown in Fig. \ref{fig:eg3_results}. The observed trend of this example is quite similar to the example given in Section \ref{sec.2-d example}. Overall, these numerical findings corroborate our theoretical analysis on convergence.

\begin{figure}[!htbp]
\centering
\captionsetup{singlelinecheck = false,labelsep=period, font=small}
\captionsetup[subfigure]{justification=centering}
\subfloat[policy error]{\includegraphics[width=0.46\linewidth]{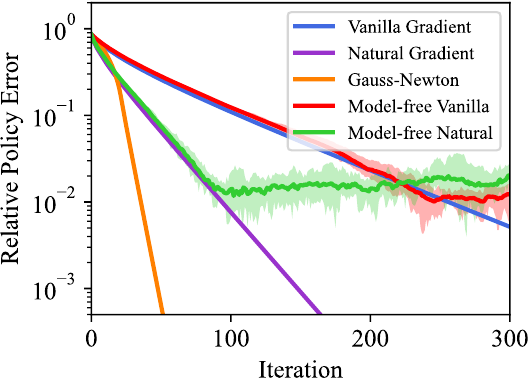}}
\subfloat[cost error]{\includegraphics[width=0.46\linewidth]{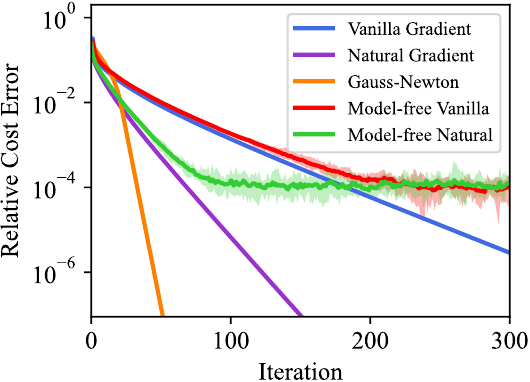}}
\caption{Learning curves of different methods for Example II. The solid lines correspond to the mean and the shaded regions correspond to the interval between maximum and minimum values over 10 runs.}
\label{fig:eg3_results}
\end{figure}

\section{Conclusion}
\label{sec:conclusion}
In this work, we have investigated the optimization landscape of three distinct policy gradient algorithms for SOF problems. Initially, we demonstrated various crucial properties of the SOF cost function, including coercivity, $L$-smoothness, and $M$-Lipschitz continuity of its Hessian. Utilizing these foundational properties, we unearthed new understandings about the convergence behaviors and rates at which all three policy gradient algorithms arrive at stationary points. These stationary points are generally influenced by the initial state distribution. Moreover, provided that the initial gain is around a local minimum, we demonstrated that the vanilla policy gradient exhibits linear convergence towards that minimum. Our numerical experiments suggest that both the vanilla policy gradient method and the natural policy gradient method can be implemented in a model-free manner, as long as the gradient estimations are sufficiently accurate. \textcolor{blue}{Additionally, recent literature highlights the potential of model-free SOF in  $H_{\infty}$ control \cite{arogeti2021static}.  Our next steps will involve expanding the convergence analysis in this specialized domain.}

\vspace{-0.5em}

\section*{Acknowledgment}
We extend our gratitude to Yang Zheng for his valuable suggestions.

\vspace{-0.5em}
\appendix
\subsection{Intermediate Lemmas}
\label{appen.intermediate lemma}
\begin{lemma}
\label{lemma.upper_bound}
The upper bound of $\|P_K\|$ and $\|\Sigma_K\|$ are given by 
\begin{equation}
\nonumber
\|P_K\|\le \frac{J(K)}{\mu}, \quad\|\Sigma_K\|\le {\rm Tr}(\Sigma_K)\le \frac{J(K)}{\sigma_{\rm min}(Q)}.
\end{equation}
\end{lemma}
\begin{proof}
From \eqref{eq.cost_in_P}, one has
\begin{equation}
\nonumber
J(K)= {\rm Tr}(P_KX_0)\ge \mu\|P_K\|.
\end{equation}
Then, the first claim can be directly obtained. Similarly, $J(K)$ can also be lower bounded by
\begin{equation}
\nonumber
\begin{aligned}
J(K)= &{\rm Tr}((Q+C^\top K^\top R K C )\Sigma_K)\ge \sigma_{\rm min}(Q){\rm Tr}(\Sigma_K)\\
\ge & \sigma_{\rm min}(Q)\|\Sigma_K\|,
\end{aligned}
\end{equation}
which leads to the second claim.
\end{proof}

\begin{lemma}
\label{lemma.K_upper_bound}
For any $K\in \mathbb{K}_\alpha$, it holds that 
\begin{equation}
\nonumber
\|KC\|\le \psi:=\frac{\sqrt{\|R\|\alpha+\|B\|^2 \alpha^2/\mu }}{\sqrt{\mu}\sigma_{\rm min}(R)}+\frac{\|B\| \|A\|\alpha}{\mu\sigma_{\rm min}(R)}.
\end{equation}
\end{lemma}

\begin{proof}
First, we can observe that
\begin{equation}
\nonumber
\begin{aligned}
\|KC\|&= \|(R+B^\top P_K B)^{-1}(R+B^\top P_K B)KC \|\\
&\le \|(R+B^\top P_K B)^{-1}\|\|(R+B^\top P_K B)KC \|\\
&\le \frac{\|(R+B^\top P_K B)KC-B^\top P_KA+B^\top P_KA \|}{\sigma_{\rm min}(R)}\\
& \le \frac{\|E_K\|+\|B^\top P_KA\|}{\sigma_{\rm min}(R)}\\
& \le \frac{\sqrt{{\rm Tr}(E_K^\top E_K)}+\|B^\top P_KA\|}{\sigma_{\rm min}(R)}.
\end{aligned}
\end{equation}

From \eqref{eq:lower_bound_2}, we know that 
\begin{equation}
\nonumber
\begin{aligned}
{\rm Tr}(E_K^\top E_K) \le \frac{\|R+B^\top P_K B\|}{\mu}J(K).
\end{aligned}
\end{equation}
Thereby, we finally have
\begin{equation}
\nonumber
\begin{aligned}
\|KC\| &\le \frac{\sqrt{\|R+B^\top P_K B\|J(K)}}{\sqrt{\mu}\sigma_{\rm min}(R)}+\frac{\|B^\top P_KA\|}{\sigma_{\rm min}(R)}\\
&\le\frac{\sqrt{\|R\|\alpha+\|B\|^2 \alpha^2/\mu }}{\sqrt{\mu}\sigma_{\rm min}(R)}+\frac{\|B\| \|A\|\alpha}{\mu\sigma_{\rm min}(R)},
\end{aligned}
\end{equation}
where the last step follows from Lemma \ref{lemma.upper_bound}.
\end{proof}

\subsection{Proof of Lemma \ref{lemma.coercivity}}
\label{append:proof_coercivity}
\begin{proof}
From \eqref{eq.cost_in_P}, we can show that
\begin{equation} 
\nonumber
\begin{aligned}
J(K_i)
&={\rm Tr}((Q+C^\top K_i^\top RK_iC)\Sigma_{K_i})\\
&\ge\mu\sigma_{\rm min}(R)\sigma_{\rm min}(C)^2\| K_i\|^2,
\end{aligned}
\end{equation}
which directly leads to that $J(K_i)\rightarrow +\infty$ as $\|K_i\|\rightarrow +\infty$.

By \eqref{eq.P_expand}, we also have
\begin{equation}   
\nonumber
\begin{aligned}
J(K_i) &={\rm Tr}(\sum_{j=0}^{\infty}{\mathcal{A}_{K_i}^\top}^j(Q+C^\top K_i^\top RK_iC){\mathcal{A}_{K_i}}^jX_0)\\
&\ge \mu \sigma_{\rm min}(Q)\sum_{j=0}^{\infty}\|{\mathcal{A}_{K_i}}^j\|_F^2\ge \mu \sigma_{\rm min}(Q)\sum_{j=0}^{\infty}\rho(\mathcal{A}_{K_i})^{2j}\\
&= \mu \sigma_{\rm min}(Q)\frac{1-\rho(\mathcal{A}_{K_i})^{\infty}}{1-\rho(\mathcal{A}_{K_i})^2}.
\end{aligned}
\end{equation}
Since $ \rho(\mathcal{A}_{K})=1$ when $K \in \partial \mathbb{K}$, by continuity of the $\rho(\mathcal{A}_{K_i})$, we have $\rho(\mathcal{A}_{K_i}) \rightarrow 1$ as $K_i \rightarrow K \in \partial \mathbb{K}$. Therefore, for every $\epsilon >0$, there exists some $N(\epsilon)\in \mathbb{N}$ such that $ 1 - \rho(\mathcal{A}_{K_i}) < \epsilon$ for all $i \ge N(\epsilon)$. That is, $1>\rho(\mathcal{A}_{K_i})>1-\epsilon$ for $i\ge N(\epsilon)$. Hence, $J(K_i)$ is bounded below by
\begin{equation} 
\nonumber
J(K_i) \ge \mu \sigma_{\rm min}(Q)\frac{1}{1-(1-\epsilon)^2}.
\end{equation}
It thus follows that
$J(K_i) \rightarrow +\infty$ as $K_i \rightarrow \partial \mathbb{K}$. This completes the proof of Lemma \ref{lemma.coercivity}.
\end{proof}

\subsection{Derivations of bounds $q_1$ and $q_2$ in Lemma \ref{lemma.L-smooth}}
\label{appendix.bound}
Firstly, for the bound $q_1$, we can easily observe that
\begin{equation}
\label{eq.upper_of_q1}
\begin{aligned}
 q_1 \le &\sup_{\|Z\|_F=1}\left(\|C^\top Z^\top (R+B^{\top} P_KB)ZC\|{\rm Tr}(\Sigma_K)\right)\\
 \le &\sup_{\|Z\|_F=1}\left( \|C\|^2\|Z\|_F^2(\|R\|+\|B\|^2\|P_K\|){\rm Tr}(\Sigma_K)\right)\\
  \le & \|C\|^2\left(\|R\|+\|B\|^2\frac{J(K)}{\mu}\right)\frac{J(K)}{\sigma_{\rm min}(Q)},
 \end{aligned}
\end{equation}
where the last step follows from Lemma \ref{lemma.upper_bound}.

Next, we focus on the upper bound of $q_2$. Using the Cauchy-Schwarz inequality, we can show that
\begin{equation}
\nonumber
\begin{aligned}
q_2\le &\sup_{\|Z\|_F=1}\left(\|(BZC)^\top P'_K[Z]\mathcal{A}_{K}\Sigma_K^{1/2}\|_F \|\Sigma_K^{1/2}\|_F\right)\\
\le& \sup_{\|Z\|_F=1}\left(\|C\|\|Z\|\|B\|\|P'_K[Z]\|\|\mathcal{A}_{K}\Sigma_K^{1/2}\|_F \sqrt{{\rm Tr}(\Sigma_K)}\right)\\
\le& \|C\|\|B\|\sup_{\|Z\|_F=1}(\|P'_K[Z]\|) \sqrt{{\rm Tr}(\mathcal{A}_{K}\Sigma_K\mathcal{A}_{K}^\top)} \sqrt{{\rm Tr}(\Sigma_K)}.
\end{aligned}
\end{equation}
By \eqref{eq.lyapunov_equation_sigma}, it is not hard to see that $\Sigma_K \succ \mathcal{A}_{K}\Sigma_K\mathcal{A}_{K}^\top$. Therefore, we further have
\begin{equation}
\label{eq.upper_of_q2} 
\begin{aligned}
q_2 &\le  \|C\|\|B\|{\rm Tr}(\Sigma_K)\sup_{\|Z\|_F=1}\|P'_K[Z]\|\\
&\le  \|C\|\|B\|\frac{J(K)}{\sigma_{\rm min}(Q)}\sup_{\|Z\|_F=1}\|P'_K[Z]\|,
\end{aligned}
\end{equation}
where the last step follows from Lemma \ref{lemma.upper_bound}. Then, the only thing left is to show the following bound holds
\begin{equation}
\nonumber
\sup_{\|Z\|_F=1}\|P'_K[Z]\| \le \zeta_1 \|P_K\|,
\end{equation}
where $\zeta_1$ is as given by \eqref{eq.zeta1}.

We will prove the above inequality by showing that $P'_K[Z] \preceq \zeta_1 P_K$. Based on \eqref{eq.P_expand} and \eqref{eq.derivative_P}, $(C^\top Z^\top E_K  + E_K^\top ZC)\preceq \zeta_1(Q+C^\top K^\top RKC)$ will directly lead to $P'_K[Z] \preceq \zeta_1 P_K$ for a given $\zeta_1 \in \mathbb{R}^+$. Now the remaining task is to find such $\zeta_1$. From \eqref{eq.lyapunov_equation}, we have
\begin{equation}
\label{eq.eata_1_derivation}
\begin{aligned}
&C^\top Z^\top E_K  + E_K^\top ZC \\
&=C^\top Z^\top RKC+C^\top K^\top RZC\\
&\qquad -C^\top Z^\top B^\top P_K\mathcal{A}_{K}-\mathcal{A}_{K}^\top P_KBZC\\
&\preceq  C^\top Z^\top RZC+C^\top K^\top RKC\\
&\quad +\mathcal{A}_{K}^\top P_K\mathcal{A}_{K}+(BZC)^\top P_KBZC\\
&= P_K - Q + C^\top Z^\top RZC + (BZC)^\top P_KBZC\\
&\preceq  \|P_K+C^\top Z^\top RZC + (BZC)^\top P_KBZC\|I-Q\\
&\preceq  \frac{Q}{\sigma_{\rm min}(Q)}\Big(\frac{\alpha}{\mu}(1+\|B\|^2\|C\|^2)+\|R\|\|C\|^2 \Big)-Q.
\end{aligned}
\end{equation}
Therefore, we prove that $P'_K[Z] \preceq \zeta_1 P_K$. According to  \eqref{eq.upper_of_q2} and Lemma \ref{lemma.upper_bound}, this directly leads to \eqref{eq.upper_of_q2_new_final}.

\subsection{Proof of Lemma \ref{lemma.gradient_dominance}}
\label{appendix.proof_dominance}

The performance difference lemma, also referred to as almost smoothness,  serves as the foundational element for establishing the gradient domination condition.

\begin{lemma}[Performance difference lemma]
\label{lemma.performance_diff}
\textnormal{Let $K$, $K' \in \mathbb{K}$. Then, the following relationship exists:
\begin{equation}
\nonumber
\begin{aligned}
J&(K')-J(K) = 2{\rm Tr}\big(\Sigma_{K'}(K'C-KC)^\top E_K\big)+\\
&+{\rm Tr}\big(\Sigma_{K'}(K'C-KC)^\top (R+B^\top P_K B)(K'C-KC)\big).
\end{aligned}
\end{equation}}
\end{lemma}

\begin{proof}
Consider state and action sequences 
 ${x_t'}$ and ${u_t'}$ generated by $K'$, and let $c_t'=x_t'^\top Qx_t'+u_t'^\top Ru_t'$. Then, one has
\begin{equation}
\nonumber
\begin{aligned}
J(K')-&J(K) \\
=& \mathbb{E}_{x_0\sim \mathcal{D}}\Big[\sum_{t=0}^{\infty}c_t' - V_K(x_0) \Big]\\
=& \mathbb{E}_{x_0\sim \mathcal{D}}\Big[\sum_{t=0}^{\infty}\left(c_t'+V_K(x_t')-V_K(x_t')\right) - V_K(x_0) \Big]\\
=& \mathbb{E}_{x_0\sim \mathcal{D}}\Big[\sum_{t=0}^{\infty}(c_t'+V_K(x_{t+1}')-V_K(x_t')) \Big],
\end{aligned}
\end{equation}
where the last step takes advantage of the fact that $x_0=x_0'$.

Let $A_K(x_t,K')=c_t+V_K(x_{t+1})-V_K(x_t)|_{u_t=-K'Cx_t}$, which can be expanded as 
\begin{equation}
\nonumber
\begin{aligned}
A_K&(x_t, K') \\
=& x_t^\top(Q + C^\top K'^\top RK'C)x_t +x_t^\top\mathcal{A}_{K'}^\top P_K\mathcal{A}_{K'}x_t - V_K(x_t)\\
=& x_t^\top(Q + (K'C-KC+KC)^\top R(K'C-KC+KC))x_t\\
& +x_t^\top(A-B(K'C-KC+KC))^\top P_K(A\\
&\qquad\qquad -B(K'C-KC+KC))x_t - V_K(x_t)\\
=&2x_t^\top(K'C-KC)^\top ((R+B^\top P_K B)KC-B^\top P_K A)x_t\\
& +x_t^\top(K'C-KC)^\top (R+B^\top P_K B)(K'C-KC)x_t\\
=&2x_t^\top(K'C-KC)^\top E_K x_t\\
& +x_t^\top(K'C-KC)^\top (R+B^\top P_K B)(K'C-KC)x_t.
\end{aligned}
\end{equation}
Then, we get that
\begin{equation}
\nonumber
\begin{aligned}
J(&K')-J(K) \\
=& \mathbb{E}_{x_0\sim \mathcal{D}}\Big[\sum_{t=0}^{\infty}A_K(x_t', K')  \Big]\\
=& \mathbb{E}_{x_0\sim \mathcal{D}}\Big[\sum_{t=0}^{\infty}\Big(2{\rm Tr}\big(x_t'x_t'^\top(K'C-KC)^\top E_K\big)+\\
&{\rm Tr}\big(x_t'x_t'^\top(K'C-KC)^\top (R+B^\top P_K B)(K'C-KC)\big)\Big)\Big]\\
=& 2{\rm Tr}\big(\Sigma_{K'}(K'C-KC)^\top E_K\big)+\\
&{\rm Tr}\big(\Sigma_{K'}(K'C-KC)^\top (R+B^\top P_K B)(K'C-KC)\big).
\end{aligned}
\end{equation}
\end{proof}

Next, we show the main proof of Lemma \ref{lemma.gradient_dominance}.
\begin{proof}
Let $X=(R+B^\top P_K B)^{-1}E_K \Sigma_{K'}C^\top\mathcal{L}_{K'}^{-1}$.
From Lemma \ref{lemma.performance_diff}, we find that
\begin{equation}
\label{eq.performance_diff_1}
\begin{aligned}
J&(K')-J(K)\\
=& 2{\rm Tr}\big(\Sigma_{K'}(K'C-KC)^\top E_K\big)\\
&+{\rm Tr}\big(\Sigma_{K'}(K'C-KC)^\top (R+B^\top P_K B)(K'C-KC)\big)\\
=& {\rm Tr}\big(\Sigma_{K'}C^\top(\Delta K+X)^\top (R+B^\top P_K B)(\Delta K+X)C\big)\\
&-{\rm Tr}\big(\Sigma_{K'}C^\top\mathcal{L}_{K'}^{-1}C\Sigma_{K'}E_K^\top (R+\\
&\qquad \qquad \qquad \qquad B^\top P_K B)^{-1} E_K\Sigma_{K'}C^\top\mathcal{L}_{K'}^{-1}C\big)\\
\ge&-{\rm Tr}\big(\mathcal{L}_{K'}^{-1}C\Sigma_{K'}E_K^\top (R+ B^\top P_K B)^{-1} E_K\Sigma_{K'}C^\top\big),
\end{aligned}
\end{equation}
where $\Delta K=K'-K$  and the equality holds when $K'=K-X$.

Then, one has
\begin{equation}
\label{eq.gradient_dominance_3}
\begin{aligned}
&J(K)-J(K^*)\\
&\le {\rm Tr}\big(\mathcal{L}_{K^*}^{-1}C\Sigma_{K^*}E_K^\top (R+B^\top P_K B)^{-1} E_K\Sigma_{K^*}C^\top\big)\\
&\le\|\Sigma_{K^*}C^\top\mathcal{L}_{K^*}^{-1}C\Sigma_{K^*}\|{\rm Tr}\big( E_K^\top (R+B^\top P_K B)^{-1} E_K \big)\\
&\le\|\Sigma_{K^*}C^\top\mathcal{L}_{K^*}^{-1}C\|\|\Sigma_{K^*}\|{\rm Tr}\big( E_K^\top (R+B^\top P_K B)^{-1} E_K \big)\\
&\le\|\Sigma_{K^*}\|{\rm Tr}\big( E_K^\top (R+B^\top P_K B)^{-1} E_K \big)\\
&\le \frac{\|\Sigma_{K^*}\|{\rm Tr}\big(E_K^\top E_K\big)}{\sigma_{\rm min}(R)}.
\end{aligned}
\end{equation}
From \eqref{eq.gradient}, it follows that 
\begin{equation}  
\label{eq.gradient_dominance_2}
\begin{aligned}
\|\nabla J(K))\|_F^2&= 4{\rm Tr}(C\Sigma_KE_K^\top E_K\Sigma_KC^\top)\\
&\ge 4\mu^2\sigma_{\rm min}(C)^2{\rm Tr}(E_K^\top E_K),\  \forall C \in \mathbb{C}.
\end{aligned}
\end{equation}
By \eqref{eq.gradient_dominance_3} and \eqref{eq.gradient_dominance_2}, one has
\begin{equation}
\label{eq.gradient_dominance_4}
J(K)-J(K^*)\le \frac{\|\Sigma_{K^*}\|\|\nabla J(K))\|_F^2}{4\mu^2\sigma_{\rm min}(C)^2\sigma_{\rm min}(R)}, \ \forall C\in \mathbb{C}.
\end{equation}

Suppose $K'$ satisfies that $K'=K-X$. According to \eqref{eq.performance_diff_1}, we get
\begin{equation}
\label{eq:lower_bound_2_ori}
\begin{aligned}
J(K)&-J(K^*)\\
&\ge J(K)-J(K')\\
&={\rm Tr}\big(\mathcal{L}_{K'}^{-1}C\Sigma_{K'}E_K^\top (R+B^\top P_K B)^{-1} E_K\Sigma_{K'}C^\top\big)\\
&\ge \frac{\mu{\rm Tr}\big(E_K^\top  E_K\big)}{\|R+B^\top P_K B\|}, \ \forall C \in \mathbb{C}.
\end{aligned}
\end{equation}

In addition, when $C\in \mathbb{C}$, since we can always identity the state $x$ by $x= C^{-1}y$, it is clear that $J(K^*)=J_s^*$ for every $C\in \mathbb{C}$. By replacing $J(K^*)$ in \eqref{eq.gradient_dominance_4} and \eqref{eq:lower_bound_2_ori} with $J_s^*$, we finally complete the proof.
\end{proof}

\subsection{Derivations of $\zeta_2$, $\zeta_3$, and $\zeta_4$ in Lemma \ref{lemma.continue_Hessian}}
\label{appendix.continue_Hessian}
From \eqref{eq.derivative_P}, it is clear that 
\begin{equation} 
\begin{aligned}
 &\frac{\partial P_{\bar{K}}}{\partial \delta}=\sum_{j=0}^{\infty}{\mathcal{A}_{\bar{K}}^\top}^j(C^\top {\Delta K}^\top E_{\bar{K}} +E_{\bar{K}}^\top \Delta KC){\mathcal{A}_{\bar{K}}}^j.
\end{aligned}
\end{equation}
Then, we can observe that
\begin{equation}
\begin{aligned}
&C^\top {\Delta K}^\top E_{\bar{K}}  + E_{\bar{K}}^\top \Delta KC \\
&=C^\top\Delta K^\top R\bar{K}C+C^\top\bar{K}^\top R\Delta K C\\
&\qquad-C^\top\Delta K^\top B^\top P_{\bar{K}}{\mathcal{A}_{\bar{K}}}-\mathcal{A}_{\bar{K}}^\top P_{\bar{K}}B\Delta KC\\
&\le 2\|C\|(\|R\|\|\bar{K}C\|+\|B\|\|P_{\bar{K}}\|\|\mathcal{A}_{\bar{K}}\|)\|\Delta K\|I\\
&\le \frac{2\|C\|Q}{\sigma_{\rm min}(Q)}(\|R\|\psi+\gamma\|B\|\frac{\alpha}{\mu})\|\Delta K\|,
\end{aligned}
\end{equation}
where the last step follows from Lemma \ref{lemma.K_upper_bound}.
Therefore, according to \eqref{eq.P_expand}, we have $\frac{\partial P_{\bar{K}}}{\partial \delta} \preceq \zeta_2 \|\Delta K\|P_{\bar{K}}$.

Next, we will prove that $\frac{\partial P'_{\bar{K}}[Z]}{\partial \delta} \preceq \zeta_3 \|\Delta K\|P_{\bar{K}}$. Based on \eqref{eq.P_derivative}, we get
\begin{equation}   
\frac{\partial P'_{\bar{K}}[Z]}{\partial \delta}= \sum_{j=0}^{\infty}{\mathcal{A}_{\bar{K}}^\top}^jS_3{\mathcal{A}_{\bar{K}}}^j,
\end{equation}
where
\begin{equation}  
\nonumber
\begin{aligned}
S_3&:=C^\top Z^\top (R+B^{\top} P_{\bar{K}}B)\Delta KC-(BZC)^\top\frac{\partial P_{\bar{K}}}{\partial \delta}{\mathcal{A}_{\bar{K}}}  \\
&\quad + C^\top \Delta K^\top  (R+B^{\top} P_{\bar{K}}B) ZC-\mathcal{A}_{\bar{K}}^\top\frac{\partial P_{\bar{K}}}{\partial \delta}BZC \\ 
&\quad-(B\Delta KC)^\top P'_{\bar{K}}[Z]{\mathcal{A}_{\bar{K}}}-\mathcal{A}_{\bar{K}}^\top P'_{\bar{K}}[Z]B\Delta KC.
\end{aligned}
\end{equation}
Recalling that $P'_{\bar K}[Z] \preceq \zeta_1 P_{\bar K}$ and $\frac{\partial P_{\bar{K}}}{\partial \delta} \preceq \zeta_2 \|\Delta K\|P_{\bar{K}}$, we can also show that
\begin{equation}
\nonumber
\begin{aligned}
S_3 &\le 2(\|C\|^2\|R\|+\|C\|^2\|B\|^2\|P_{\bar{K}}\|+\zeta_2\gamma\|B\|\|C\|\|P_{\bar{K}}\|\\
&\qquad+\zeta_1\gamma\|B\|\|C\|\|P_{\bar{K}}\|)\|\Delta K\|I\\
&\le \frac{2\|C\|Q}{\sigma_{\rm min}(Q)}\big(\|C\|\|R\|+\|B\|(\|C\|\|B\|\\
&\qquad\qquad\qquad\qquad+\zeta_1\gamma+\zeta_2\gamma)\frac{\alpha}{\mu}\big)\|\Delta K\|.
\end{aligned}    
\end{equation}
Therefore, we get $\frac{\partial P'_{\bar{K}}[Z]}{\partial \delta} \preceq \zeta_3 \|\Delta K\|P_{\bar{K}}$.

Similarly, for $P''_{\bar{K}}[Z]$, from \eqref{eq.second_derivative_P}, we can show that
\begin{equation}
\nonumber
\begin{aligned}
S_1&\le 2(\|C\|^2\|R\|+\|C\|^2\|B\|^2\|P_{K}\|+\zeta_1\gamma\|B\|\|C\|\|P_{K}\|)I\\
&\le \frac{2\|C\|Q}{\sigma_{\rm min}(Q)}\big(\|C\|\|R\|+\|B\|(\|C\|\|B\|+\zeta_1\gamma)\frac{\alpha}{\mu}\big).\\
\end{aligned}    
\end{equation}
So, it is clear $P''_{\bar{K}}[Z]\preceq \zeta_4 P_{\bar{K}} $, which completes the derivations.

\subsection{Proof of Theorem \ref{theorem: natural_pg}}
\label{appendix:natural_pg}

\begin{proof}
We can easily modify the proof of Theorem \ref{theorem: first_order_convergence} to show that for every $K \in \mathbb{K}_{\alpha}$, if $\eta \le \mu \sigma_{\rm min}(C)^2/L$, the line segment $[K,K-\eta \nabla^{\rm NA} J(K)] \subseteq \mathbb{K}_{\alpha}$. Then from \eqref{eq.L-smooth}, one has 
\begin{equation}
\nonumber
\begin{aligned}
&J(K_{i+1})\\
&\le J(K_i)-\eta{\rm Tr}(\nabla J(K_i)^\top\nabla^{\rm NA} J(K_i) )+\frac{\eta^2L}{2}\|\nabla^{\rm NA} J(K_i) \|_F^2\\
&\le J(K_i)-\eta(\sigma_{\rm min}(\mathcal{L}_{K_i})-\frac{\eta L}{2})\|\nabla^{\rm NA} J(K_i) \|_F^2\\
&\le J(K_i)-\eta(\mu\sigma_{\rm min}(C)^2-\frac{\eta L}{2})\|\nabla^{\rm NA} J(K_i) \|_F^2\\
&\le J(K_i)-\frac{ \eta\mu\sigma_{\rm min}(C)^2  }{2}\|\nabla^{\rm NA} J(K_i) \|_F^2,
\end{aligned}
\end{equation}
where the last inequality takes into account that $\eta \le \mu \sigma_{\rm min}(C)^2/L$. Note that the boundary $\mu \sigma_{\rm min}(C)^2/L$ is selected for achieving the fastest convergence rate. By summing up the above inequality, one has
\begin{equation}
\nonumber
\frac{ \mu\sigma_{\rm min}(C)^2  \eta}{2}\sum_{i=0}^{N}\|\nabla^{\rm NA} J(K_i) \|_F^2 \le J(K_0) - J(K^{\star}).
\end{equation}
Consequently, it follows that
\begin{equation}
\nonumber
\min_{0\le i\le N} \|\nabla^{\rm NA} J(K_i) \|_F^2 \le  \frac{2\alpha}{ \eta\mu\sigma_{\rm min}(C)^2   N}.
\end{equation}
Thus, the natural policy gradient method can attain an $\epsilon$-stationary point in $\frac{2\alpha}{\eta\mu\sigma_{\rm min}(C)^2 \epsilon^2}$ iterations.

When $C\in \mathbb{C}$, by \eqref{eq.gradient} and \eqref{eq.L-smooth}, one has
\begin{equation}
\nonumber
\begin{aligned}
J(K_{i+1})&\le J(K_i)-4\eta{\rm Tr}(\Sigma_{K_i} E_{K_i}^\top E_{K_i})+2\eta^2L\|E_{K_i} C^{-1}\|_F^2\\
& \le J(K_i)-4\eta(\mu-\frac{ L\eta}{2\sigma_{\rm min}(C)^2})\| E_{K_i}\|_F^2\\
& \le J(K_i)-2 \mu \eta\| E_{K_i}\|_F^2\\
& \le J(K_i)-\frac{ 2\eta \mu\sigma_{\rm min}(R)}{\|\Sigma_{K^{\star}}\|}(J(K_i)-J(K^{\star})),
\end{aligned}
\end{equation}
where the last step follows from \eqref{eq.gradient_dominance_1}.
It directly follows that 
\begin{equation}
\nonumber
J(K_{i})-J_s^{\star} \le \left(1-\frac{2\eta\mu\sigma_{\rm min}(R)}{\|\Sigma_{K^{\star}}\|}\right)^i(J(K_0)-J_s^{\star}),
\end{equation}
which completes the proof of the second claim.
\end{proof}

\subsection{Proof of Theorem \ref{theorem: GN_convergence}}
\label{appendix:proof_GN}

\begin{proof}
We can easily modify the proof of Theorem \ref{theorem: first_order_convergence} to show that for every $K \in \mathbb{K}_{\alpha}$, if $\eta \le \mu\sigma_{\rm min}(R)\sigma_{\rm min}(C)^2/L$, the segment $[K,K-\eta \nabla^{GN} J(K)] \subseteq \mathbb{K}_{\alpha}$. Then from \eqref{eq.L-smooth}, one has 
\begin{equation}
\nonumber
\begin{aligned}
&J(K_{i+1})\\
&\le J(K_i)-\eta{\rm Tr}(\nabla J(K_i)^\top\nabla^{\rm GN} J(K_i))+\frac{\eta^2L}{2}\|\nabla^{\rm GN} J(K_i)\|_F^2\\
&\le J(K_i)-\eta(\mu\sigma_{\rm min}(R)\sigma_{\rm min}(C)^2-\frac{\eta L}{2})\|\nabla^{\rm GN} J(K_i)\|_F^2\\
&\le J(K_i)-\frac{\eta \mu\sigma_{\rm min}(R)\sigma_{\rm min}(C)^2}{2} \|\nabla^{\rm GN} J(K_i)\|_F^2,
\end{aligned}
\end{equation}
where the last inequality considers the boundary of step size, i.e., $\eta \le \mu\sigma_{\rm min}(R)\sigma_{\rm min}(C)^2/L$. The boundary $\mu\sigma_{\rm min}(R)\sigma_{\rm min}(C)^2/L$ is selected for achieving the fastest convergence rate. By summing up the above inequality, one has
\begin{equation}
\nonumber
\frac{\eta \mu\sigma_{\rm min}(R)\sigma_{\rm min}(C)^2}{2}\sum_{i=0}^{N}\|\nabla^{\rm NA} J(K_i) \|_F^2 \le J(K_0) - J(K^{\star}).
\end{equation}
Consequently, it follows that
\begin{equation}
\nonumber
\min_{0\le i\le N} \|\nabla^{\rm NA} J(K_i) \|_F^2 \le  \frac{2\alpha}{ \eta \mu\sigma_{\rm min}(R)\sigma_{\rm min}(C)^2 N}.
\end{equation}
Thus, the Gauss-Newton method can attain an $\epsilon$-stationary point in $\frac{2\alpha}{\eta \mu\sigma_{\rm min}(R)\sigma_{\rm min}(C)^2\epsilon^2}$ iterations.

When $C\in \mathbb{C}$, by \eqref{eq.gradient} and \eqref{eq.L-smooth}, one has
\begin{equation}
\nonumber
\begin{aligned}
J(K_{i+1})&\le J(K_i)-4\eta{\rm Tr}(\Sigma_{K_i} E_{K_i}^\top(R+B^\top P_{K_i} B)^{-1} E_{K_i})\\
&\qquad\qquad+2\eta^2L\|(R+B^\top P_{K_i} B)^{-1} E_{K_i} C^{-1}\|_F^2\\
&\le J(K_i)-4\eta(\mu-\frac{\eta L}{2\sigma_{\rm min}(R) \sigma_{\rm min}(C)^2})\times\\
&\qquad\qquad\qquad{\rm Tr}(E_{K_i}^\top(R+B^\top P_{K_i} B)^{-1} E_{K_i})\\
&\le J(K_i)-2\eta\mu{\rm Tr}(E_{K_i}^\top(R+B^\top P_{K_i} B)^{-1} E_{K_i})\\
&\le J(K_i)-\frac{2\eta\mu}{\|\Sigma_{K^{\star}}\|}(J(K_i)-J(K^{\star})).
\end{aligned}
\end{equation}
where the last step follows from \eqref{eq.gradient_dominance_1}. It directly follows that 
\begin{equation}
\nonumber
J(K_{i})-J_s^{\star} \le \big(1-\frac{2\eta\mu}{\|\Sigma_{K^{\star}}\|}\big)^i(J(K_0)-J_s^{\star}),
\end{equation}
which completes the proof of the second claim.
\end{proof}


\ifCLASSOPTIONcaptionsoff
  \newpage
\fi

\bibliographystyle{ieeetr}
\bibliography{ref}

\vskip -3\baselineskip plus -1fil
\begin{IEEEbiography}[{\includegraphics[width=1in,height=1.25in,clip,keepaspectratio]{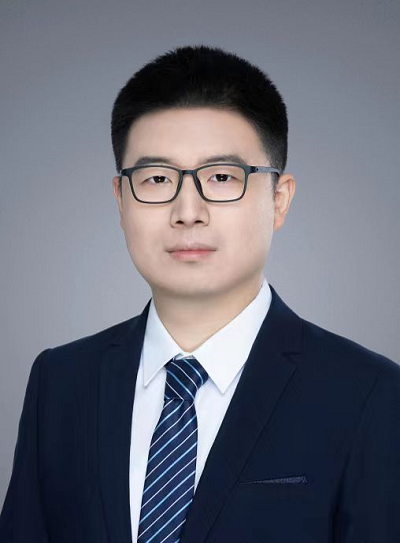}}]{Jingliang Duan} received his doctoral degree in mechanical engineering from the School of Vehicle and Mobility at Tsinghua University, China, in 2021. In 2019, he spent time as a visiting student researcher in the Department of Mechanical Engineering at the University of California, Berkeley. Following his Ph.D., he served as a research fellow in the Department of Electrical and Computer Engineering at the National University of Singapore from 2021 to 2022. He is currently a tenured associate professor in the School of Mechanical Engineering, University of Science and Technology Beijing, China. His research interests include reinforcement learning, optimal control, and self-driving decision-making.
\end{IEEEbiography}
\vskip -3\baselineskip plus -1fil
\begin{IEEEbiography}[{\includegraphics[width=0.9in,height=1.25in,clip,keepaspectratio]{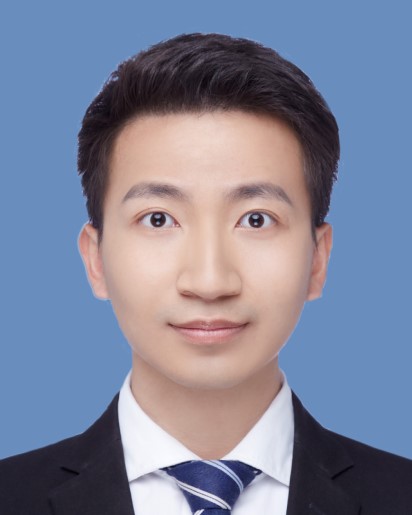}}]{Jie Li}
received the B.S. degree in automotive engineering from Tsinghua University, Beijing, China, in 2018. He is currently pursuing the Ph.D. degree in the School of Vehicle and Mobility, Tsinghua University, Beijing, China. His current research interests include model predictive control, adaptive dynamic programming and robust reinforcement learning. 
\end{IEEEbiography}
\vskip -3\baselineskip plus -1fil
\begin{IEEEbiography}[{\includegraphics[width=1in,height=1.25in,clip,keepaspectratio]{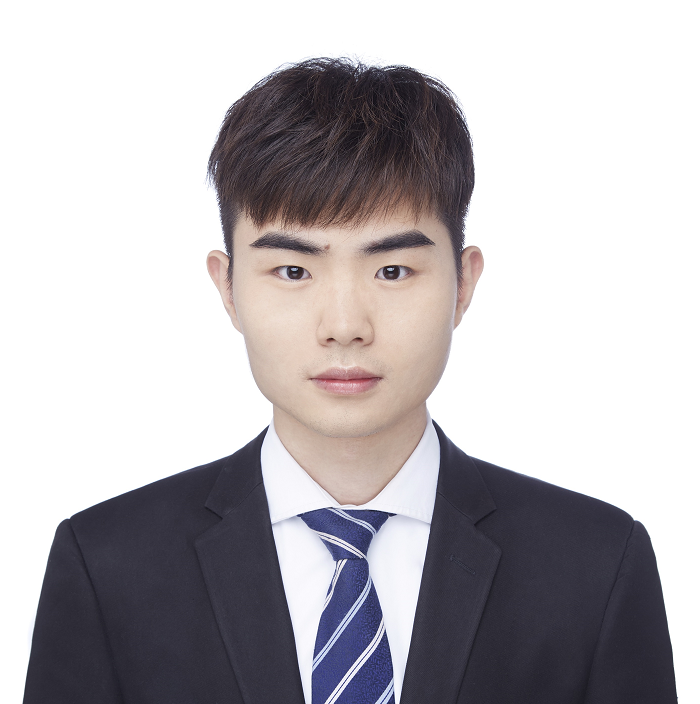}}]{Xuyang Chen}
received the B.S. degree from the honors college, Beihang University, Beijing, China, in 2019. He is currently a Ph.D. student in the Department of Electrical and Computer Engineering, National University of Singapore, Singapore. His research interests include reinforcement learning, Markov decision process and policy gradient methods.
\end{IEEEbiography}
\vskip -3\baselineskip plus -1fil
\begin{IEEEbiography}[{\includegraphics[width=1in,height=1.25in,clip,keepaspectratio]{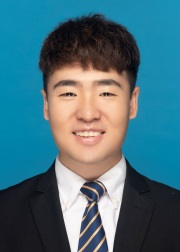}}]{Kai Zhao}
received the Ph.D. degree in control theory and control engineering from Chongqing University, Chongqing, China, in 2019. He was a Postdoctoral Fellow with the Department of Computer and Information Science, University of Macau, Macau, China, from 2019 to 2021. He is currently a Research Fellow with the Department of Electrical and Computer Engineering, National University of Singapore, Singapore. His research interests include adaptive control and prescribed performance control.
\end{IEEEbiography}

\vskip -3\baselineskip plus -1fil

\begin{IEEEbiography}[{\includegraphics[width=1in,height=1.25in,clip,keepaspectratio]{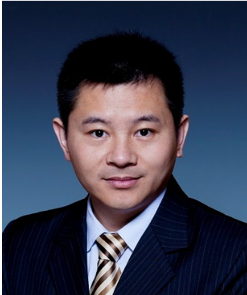}}]{Shengbo Eben Li}
(SM'16) received the M.S. and Ph.D. degrees from Tsinghua University in 2006 and 2009. He worked at Stanford University, University of Michigan, and University of California, Berkeley. He is currently a tenured professor at Tsinghua University. His active research interests include intelligent vehicles and driver assistance, reinforcement learning and distributed control, optimal control and estimation, etc.
\end{IEEEbiography}

\vskip -3\baselineskip plus -1fil
\begin{IEEEbiography}[{\includegraphics[width=1in,height=1.25in,clip,keepaspectratio]{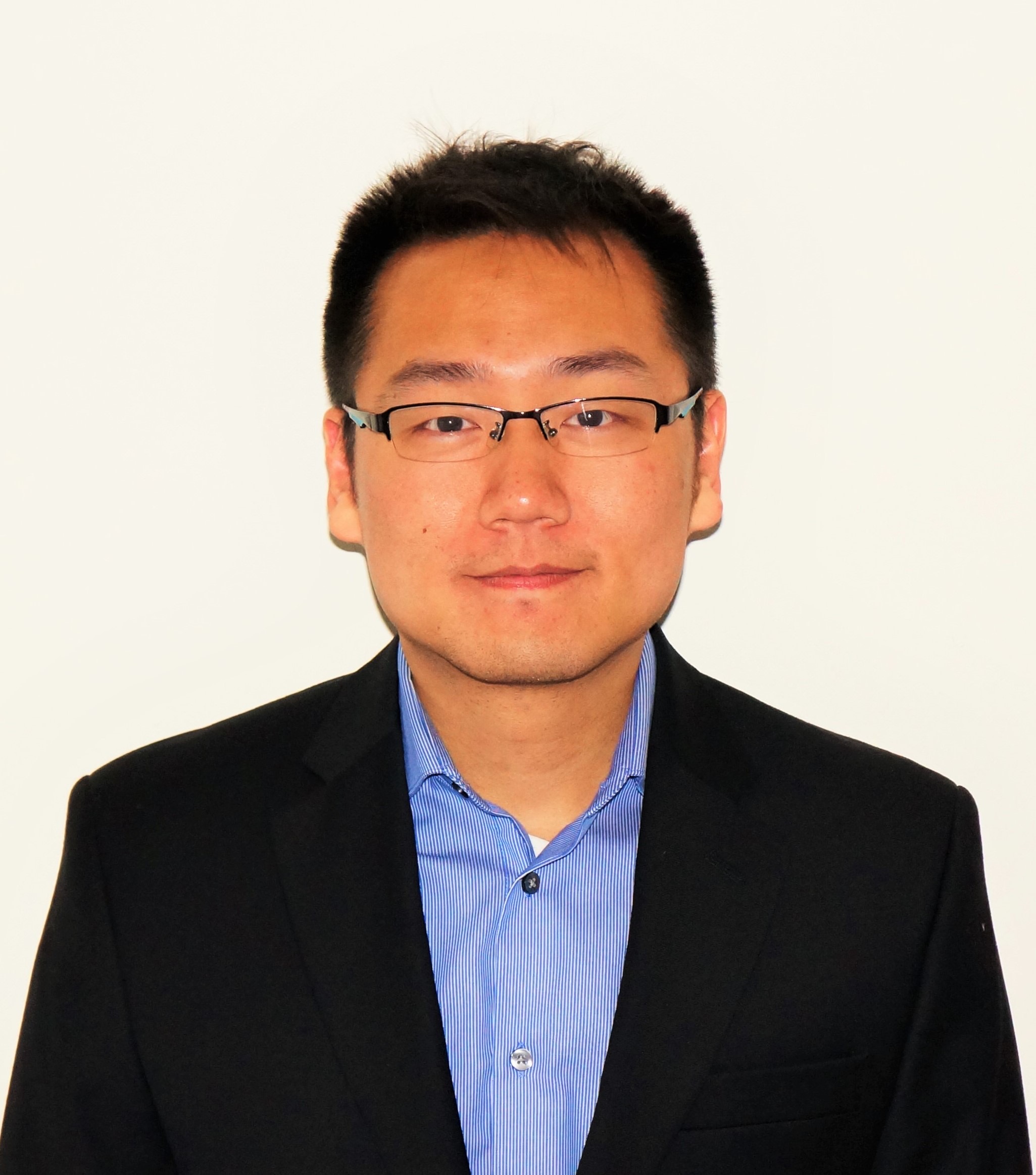}}]{Lin Zhao} received a B.S. and an M.S. degree in automatic control from the Harbin Institute of Technology, Harbin, China, in 2010 and 2012, respectively, an M.S. degree in mathematics and a Ph.D. degree in electrical and computer engineering from The Ohio State University, Columbus, OH, USA, in 2017. From 2018 to early 2020, he was a research scientist at the Aptiv Pittsburgh Technology Center (now Motional), Pittsburgh, PA, USA. He is currently an Assistant Professor in the Department of Electrical and Computer Engineering, National University of Singapore. He serves on the Young Editorial Board of the Springer Journal of Systems Science and Complexity, served as the Program Co-chair for the 17th IEEE International Conference on Control and Automation (ICCA 2022), and as Publicity Co-chair for the 62nd IEEE Conference on Decision and Control (CDC 2023). His current research focuses on control and reinforcement learning with applications in robotics.
\end{IEEEbiography}

\end{document}